\documentclass[twoside,leqno,symbols-for-thanks]{rmi}
\usepackage[colorlinks,linkcolor=black,citecolor=black,urlcolor=black, linktocpage=true,dvips]{hyperref}
\usepackage{srcltx}
\usepackage[english]{babel}
\numberwithin{equation}{section}
\setinitialpage{1}
\title[Upper bounds on the spectral function of homogeneous spaces]
{Upper bounds for the spectral function  
on homogeneous spaces via volume growth}

\author[C. Judge and R. Lyons]{Chris Judge and Russell Lyons}

\address[cjudge2@gmail.com]{Department of Mathematics, 831 E. 3rd St., Indiana University, Bloomington, IN 47405-7106 USA} 
\address[rdlyons@indiana.edu]{Department of 
Mathematics, 831 E. 3rd St., Indiana University, Bloomington, IN 47405-7106 USA} 

\amsclassification[58J50, %Spectral problems; spectral geometry; scattering theory 
58C40, 43A85,
53C30%Homogeneous manifolds
]{35P15%Estimation of eigenvalues, upper and lower bounds
}


\usepackage{amsthm}  % This is claimed to be loaded, but it is not.
\usepackage{bbm}
\usepackage{mathrsfs}
\usepackage{microtype}
\usepackage[margin=10pt,font=small,labelfont=bf]{caption}

\newtheorem{thm}{Theorem}[section]

\newtheorem{coro}[thm]{Corollary}
\newtheorem{lem}[thm]{Lemma}
\newtheorem{prop}[thm]{Proposition}

\theoremstyle{definition}

\newtheorem{defn}[thm]{Definition}
\newtheorem{eg}[thm]{Example}

\theoremstyle{remark}

\newtheorem{remk}[thm]{Remark}

\newcommand{\tb}{\widetilde{b}}

\newcommand{\Sbb}{ {\mathbb S}}
\newcommand{\Rbb}{ {\mathbb R}}
\newcommand{\Zbb}{ {\mathbb Z}}
\newcommand{\Nbb}{ {\mathbb N}}

\newcommand{\Bcal}{ {\mathcal B}}
\newcommand{\Hcal}{ {\mathcal H}}

\newcommand{\Hbb}{{\mathbb H}}
\newcommand{\Obb}{{\mathbb O}}

\newcommand{\dmu}{d \mu}

\newcommand{\und}{\frac{1}{2}}

\newcommand\sfrac[2]{#1/#2}

\providecommand{\norm}[1]{\left\lVert#1\right\rVert}

\numberwithin{equation}{section}

\renewcommand\b[1]{{\bf #1}}

\newcommand\cL{\Delta}

\DeclareMathOperator{\dist}{{dist}}

\DeclareMathOperator{\vol}{\mathsf{vol}}

\newcommand\dfn[1]{\textit{\textbf{#1}}}

\newcommand\iprod[1]{\langle #1 \rangle}
\newcommand\Bigiprod[1]{\Bigl\langle #1 \Bigr\rangle}

\newcommand\R{{\mathbb R}}
\newcommand\C{{\mathbb C}}

\newcommand\mg{g}
\newcommand\ef{\varphi}
\newcommand\cef{\overline\varphi}
\newcommand\ct{N}

\newcommand\st{\,;\;}
\newcommand\cbuldot{{\raise.25ex\hbox{$\scriptscriptstyle\bullet$}}}
\newcommand\Seq[1]{\langle #1 \rangle}
\newcommand\dom{\mathscr{D}}

\newcommand\qdom{\mathscr{Q}}

\begin{document}

%%%%%%%%%%%%%%%%%%%%%%%%%%%%%%%%%
%
%  Insert the abstract
%
%%%%%%%%%%%%%%%%%%%%%%%%%%%%%%%%%

\begin{abstract}
We use spectral embeddings to give upper bounds on the
spectral function of the
Laplace--Beltrami operator on homogeneous spaces
in terms of the volume growth of balls. In the case of compact manifolds,
our bounds extend the 1980 lower bound of Peter Li \cite{Li} for the smallest
positive eigenvalue to all eigenvalues. We also improve Li's bound itself.
Our bounds translate to explicit upper
bounds on the heat kernel for both compact and noncompact homogeneous spaces.
\end{abstract}

%%%%%%%%%%%%%%%%%%%%%%%%%%%%%%%%%
%
% Body of the article
% Please insert here the TeX source of your paper
% (except the bibliography)
%
%%%%%%%%%%%%%%%%%%%%%%%%%%%%%%%%%

\section{Introduction}

The spectrum of the (nonnegative) Laplace--Beltrami operator $\Delta_g$ associated to a 
compact\footnote{For us, the  
term ``compact manifold" means that there is no boundary. 
Note that homogeneous spaces have no boundary.} Riemannian
manifold $(M,g)$ consists of a discrete set of nonnegative eigenvalues,
each with a finite-dimensional eigenspace.  These eigenvalues rarely admit 
explicit computation, and hence much work has been devoted to their estimation (see, for example, \cite{GN13,Berger}).
To bound the
eigenvalues of $\Delta_g$ in the present paper, we adapt the method of spectral embedding
that was used in \cite{Shayan} to
estimate eigenvalues of the discrete Laplacian on a graph. See \cite{Port} for a history of other uses of
spectral embeddings of manifolds. In particular, 
we provide new lower bounds on the eigenvalues of the Laplacian on
Riemannian manifolds with transitive isometry groups, as well as upper
bounds on their heat kernels.
All of our bounds come with fully explicit constants.

The estimation of the eigenvalues of $\Delta_g$ is equivalent to the
estimation of the \dfn{eigenvalue counting function}, $\ct(\lambda)$. Here, $\ct(\lambda)$ denotes the dimension of $\Hcal_{\lambda}$, the direct sum
of the eigenspaces with eigenvalue at most $\lambda$.  
Denote by 
$V(r)$ the volume of a ball of radius $r$ in $M$ when $M$ is homogeneous (whether or not $M$ is compact). 
In the compact case, our main result is the following bound:

\begin{thm}  \label{thm:main} {\rm (Theorem \ref{thm:transitiveeig})}\enspace
If $(M,g)$ is compact homogeneous space, then for each 
$\alpha \in\, ( 0, \pi/2)$, 
\[    \ct(\lambda)~ \leq~  
\frac{\vol_g(M)}{\int_0^{\pi/2} V(\theta/\sqrt \lambda) \sin(2\theta)
\,d\theta}~
\le~
\frac{\vol_g(M)}{ \cos^2 \alpha \cdot V(\alpha \cdot \lambda^{-\sfrac{1}{2}})} 
\,.
\]
\end{thm} 

Peter Li \cite{Li} showed that the first nonzero eigenvalue
of a general compact homogeneous space satisfies
\begin{equation}  \label{eq:transitivegap}
 \sqrt{\lambda_1}~ >~ \frac{\pi}{2  D} \,,
\end{equation}
where $D$ is the diameter of $(M,g)$.
Thus, if $\alpha < \pi/2$, then $\lambda \to V(\alpha \lambda^{-\frac{1}{2}})$ is
decreasing for $\lambda\geq \lambda_1$.  
Consequently, Theorem  \ref{thm:main} implies that if $0< \alpha< \pi/2$, then
\begin{equation} \label{eq:lambdakbound}
\sqrt{\lambda_k}~ \geq~ \frac{\alpha}{V^{-1}\biggl(\frac{\textstyle 
\vrule width 0pt depth 2pt \vol
M}{\textstyle (k+1) \cdot\cos^{\vrule width 0pt height 7pt 2} \alpha }\biggr)}\,,
\end{equation}
where $0=\lambda_0 < \lambda_1 \leq \lambda_2 \leq \lambda_3 \le \cdots$
are the eigenvalues listed in ascending order with multiplicity equal to the dimension 
of the corresponding eigenspace.

In general, for small $r$, we have $V(r)= \omega_d \cdot r^d + O(r^{d+2})$,
where $d$ is the dimension of $M$ and $\omega_d$ is the volume of the unit ball in $\R^d$.   
Thus, it is natural to assume that $V(r) \geq c \cdot r^d$ for $r\le r_0$. 
In such a case, Theorem \ref{thm:main} yields 
the following bound:

\begin{coro} \label{coro:volume-ass}
Let $(M,g)$ be a compact homogeneous space,
and let $r_0$ and $c$ be positive numbers.  
If $V(r) \geq c \cdot r^d$ for $r\le r_0$, then for $\lambda \ge (\pi/2r_0)^2$,
we have
\[  \ct(\lambda)~  \leq~ \frac{\vol(M)}{ c \cdot m_d} \cdot
\lambda^{\sfrac{d}{2}}\,,
\]
where $m_d := \int_0^{\pi/2} \theta^d \sin(2\theta) \,d\theta$.
\end{coro}

Explicit $c$ and $r_0$ can be derived from an upper bound $\kappa$
on the sectional curvatures of $(M,g)$ and from the injectivity radius 
${\rm inj}(M,g)$. Indeed, for $r \leq {\rm inj}(M,g)$, 
we have $V(r) \geq V_{\kappa}(r)$, where 
$V_{\kappa}(r)$ is the volume of the ball of radius $r$ in the simply 
connected homogeneous space of constant curvature $\kappa$  \cite{Gunther}.\footnote{
    See also Theorem 3.10 in \cite{GHL}.} 
In turn, the sectional curvatures of a homogeneous space can be computed 
in terms of the Lie algebra of its isometry group 
(see, for example, \cite{Cheeger-Ebin}).
The function $V_{\kappa}(r)$ can, of course, be computed explicitly; 
see, for example, \cite[\S 3.H]{GHL} or Section \ref{sec:examples} here. 

The bound in Corollary \ref{coro:volume-ass} should be compared to  Weyl's Law,
the well-known large-$\lambda$ asymptotics of $\ct(\lambda)$:
\begin{equation} \label{eqn:Weyl}
\lim_{\lambda \to \infty} \frac{\ct(\lambda)}{\lambda^{\sfrac{d}{2}}}~ =~
 \frac{\omega_d \cdot \vol(M)}{ (2 \pi)^d}\,.
\end{equation}
Since $V(r) \sim \omega_d \cdot r^d$ as $r$ tends to zero,
Corollary \ref{coro:volume-ass} yields that 
\begin{equation} \label{est:almost-Weyl}
  \limsup_{\lambda \to \infty} \frac{\ct(\lambda)}{\lambda^{\sfrac{d}{2}}}~ \le ~
     \frac{\vol(M)}{\omega_d \cdot m_d}  \,.
\end{equation}
If the dimension $d$ is large, then the right-hand side of (\ref{est:almost-Weyl}) is much larger
than the right-hand side of (\ref{eqn:Weyl}). Indeed, a 
straightforward argument shows that for large $d$, we have 
\[ m_d~ \sim~  \frac{\pi^2}{2 d^2} \cdot
\left(\frac{\pi}{2}\right)^{d}\,,\]
and it is well known that $\omega_d \sim (d \pi)^{-\und}
(2 \pi e/d)^{\sfrac{d}{2}}$, where $e=\exp(1)$.

On the other hand, Weyl's law itself does not provide an upper 
or lower bound on $\ct(\lambda)$. Indeed, although the difference 
between the left- and right-hand sides of (\ref{eqn:Weyl}) is 
known to be at most $O(\lambda^{-\und})$ \cite{Hormander}, 
the constant is not explicit. 

Li and Yau  \cite{Li-Yau} and Gromov \cite{Gromov}
have established upper bounds on the eigenvalue counting function
of general compact Riemannian manifolds of the form
$\ct(\lambda) \leq c \cdot \lambda^{\frac{d}{2}}$ for 
$\lambda \geq \lambda_1$.
See Example \ref{eg:spheres} for a comparison of estimates
in the case that $M$ is the standard 2-dimensional sphere.

In the present paper, we also establish the following improvement of 
Li's bound on the first nonzero eigenvalue:
\begin{thm} \label{thm:intro-Li}
{\rm (Theorem \ref{thm:Li-improved})}\enspace
If $(M,g)$ is a connected, compact, Riemannian homogeneous space, then
\[
\sqrt{\lambda_1}~
>~
\frac{\pi}{2D}~ +~
\frac1D \arcsin \frac{V(D/2)}{2 \bigl(\vol M - V(D/2)\bigr)}\,.
\]
\end{thm}

If $(M,g)$ can be realized as a quotient of a Lie group with a 
bi-invariant metric,
then $M$ is called a {\it normal homogeneous space}. In theory, using the
Peter--Weyl theorem, one can compute the spectrum of the Laplacian associated
to a particular compact normal homogeneous space 
(see, for example, section 2 of \cite{BertiProc}).  
Among the much larger class of nonnormal compact homogeneous spaces, 
we are aware of only four examples whose spectrum has been computed. 
These are the cubic isoparametric minimal hypersurfaces of dimensions 3, 6, 12, and 24 that lie in spheres \cite{SolI,SolII}.

Theorem \ref{thm:main} is a consequence of an estimate 
that holds for all connected homogeneous spaces, including 
those that are not compact. In this larger
setting, rather than counting eigenvalues, we estimate the diagonal 
of the \dfn{spectral function} of the Laplacian. The spectral function,
$e_\lambda(x, y)$, is the smooth\footnote{
For us, ``smooth'' will always mean $C^\infty$.}   % FOOTNOTE
integral kernel for the orthogonal
projection $E[0, \lambda]$ onto $\Hcal_\lambda$, where $E$ is the
resolution of the identity for $\Delta$. 
In the compact case,
\begin{equation} \label{defn:integral-kernel}
e_{\lambda}(x,y)~ =~  \sum_{b \in \Bcal_\lambda}  b(x) \cdot b(y)\,, 
\end{equation}
where $\Bcal_\lambda \subset C^{\infty}(M)$ is any real 
orthonormal basis for $\Hcal_{\lambda}$.
Moreover, if $M$ is homogeneous as well as compact, then integration
gives that 
\begin{equation} \label{eq:int-spec}
\ct(\lambda)
~=~ \vol_g(M)  \cdot e_\lambda(x, x)
\end{equation}
for each $x \in M$,
and so the diagonal $e_\lambda(x, x)$ is a natural replacement for $\ct(\lambda)$
in the noncompact homogeneous setting. 

In this larger context, our main result is the following:

\begin{thm}  \label{thm:2ndmain}
{\rm (Theorem \ref{thm:homogeneous-noncompact})}\enspace
If $(M,g)$ is a connected, Riemannian homogeneous space, then 
for each $x \in M$,
$\>\alpha\in\, (0,\pi/2)$, and $\lambda > 0$, we have
\begin{equation*}
e_\lambda(x, x)~ 
\le~
\frac{1}{\int_0^{\pi/2} V(\theta/\sqrt \lambda) \sin(2\theta)
\,d\theta}
~\leq~ 
\frac{1}{\cos^2 \alpha\cdot V(\lambda^{-1/2} \cdot \alpha)}\,.
\end{equation*}
\end{thm}

By making a natural assumption
on the volume growth of small balls, we obtain the following.

\begin{coro} \label{coro:volume-ass-gen}
Let $(M,g)$ be a connected, Riemannian homogeneous space, and let $r_0$ and $c$ be positive numbers.  
If $V(r) \geq c \cdot r^d$ for $r\le r_0$, then for $\lambda\ge (\pi/2r_0)^2$,
we have
\[  e_\lambda(x, x)~  \leq~ \frac{1}{ c \cdot m_d} \cdot
\lambda^{\sfrac{d}{2}}\,,
\]
where $m_d := \int_0^{\pi/2} \theta^d \sin(2\theta) \,d\theta$.
\end{coro}

Just as in the compact case,
the bound in Corollary \ref{coro:volume-ass-gen} can be compared to
the well-known large-$\lambda$ asymptotics
\begin{equation} \label{eqn:Weyl-gen}
\lim_{\lambda \to \infty} \frac{e_\lambda(x,x)}{\lambda^{\sfrac{d}{2}}}~ =~
 \frac{\omega_d}{ (2 \pi)^d}\,,
\end{equation}
where, again, although the difference 
between the left- and right-hand sides of (\ref{eqn:Weyl-gen}) is 
known to be at most $O(\lambda^{-\und})$ \cite{Hormander}, 
the constant is not explicit. 

Because bounds on the heat kernel yield bounds on the spectral function 
(see Remark \ref{rem:heat-bounds-spectral}), most of
our bounds are new only in that constants are explicit. In addition, our proofs are very short
while not using any sophisticated mathematics.

We establish our main results by
analyzing the geometry of the map $F^\lambda \colon x \mapsto F^\lambda_x$
from $M$ to $L^2(M, \mu)$,
where $F^\lambda_x(y) := e_\lambda(x,y)$. Here is a sketch of our argument:
When $M$ is homogeneous, the image of $F^\lambda$ lies in a sphere of radius
$\rho := \rho_\lambda := \norm{F^\lambda_x} = \sqrt{e_\lambda(x, x)}$; we seek an upper bound on $\rho$.
A curve in $M$ is mapped via $F^\lambda$ to a curve in that sphere. At each
$x \in M$, the
norm of the gradient of $F^\lambda$ is at most $\sqrt\lambda \cdot \rho$.
This allows us to bound the angle $\psi_\lambda(x, y)$ between $F^\lambda_x$ and
$F^\lambda_y$ in terms of $\dist(x, y)$;  of course,
$\rho^2 \cdot \cos \psi_\lambda(x, y) = \iprod{F^\lambda_x, F^\lambda_y} = F^\lambda_x(y) = e_\lambda(x, y)$. 
%The upshot is that
%\[
%\sqrt \lambda \cdot \dist(x, y)~ \ge~ 
%\arccos \frac{e_\lambda(x, y)}{e_\lambda(x, x)}
%\,.
%\]
%Integration of $F^\lambda_x(y)^2$ yields $\rho^2$.
Therefore, $1 = \rho^2 \int_M \cos^2 \psi_\lambda(x, y) \,d\mu(y)$. The gradient bound on $F^\lambda$ supplies an upper bound on $\psi_\lambda(x, y)$ and thus a lower bound on the preceding integral. This implies an upper bound on $\rho$, as desired.
%Combining this with the preceding inequality gives our result.

We end this introduction with an outline of the remainder of the paper.
In \S \ref{sec:prelim}, we recall basic facts about the Laplacian in the
compact case, introduce the spectral embedding $F^\lambda_x(\cdot) =
e_{\lambda}(x,\cdot)$, and prove a simple gradient bound on $F^\lambda$.
In \S \ref{sec:transitive}, we
analyze the length of the image of a geodesic under $F^\lambda$ to prove
Theorem \ref{thm:main}.
We also use our method to give
a very short proof of Li's bound \eqref{eq:transitivegap}
in Remark  \ref{remk:Libound}, after which we improve Li's bound to Theorem
\ref{thm:intro-Li}.
In \S \ref{sec:noncompact}, we consider noncompact
homogeneous spaces and prove Theorem \ref{thm:2ndmain}.
Although our proofs there hold in the compact case as well, we devote the preceding separate section to compact spaces since the proofs for them are significantly more straightforward.
In \S \ref{sec:spectrals},
we apply Theorem \ref{thm:2ndmain} to obtain upper estimates 
on the heat kernel. In \S \ref{sec:misc}, we compare our method to the essence of
Li's method. In \S \ref{sec:examples}, we consider the counting
functions of the three constant curvature space forms, as well as complex hyperbolic spaces. 
Appendix \ref{sec:appendix} contains some basic facts about the spectral function.

We thank the referee for a careful reading of the manuscript.

%%%%%%%%%%%%%%%%%%%%%%%%%%%%%%%%%%%%%%%%%%%%

\section{The Laplacian and the spectral embedding}\label{sec:prelim}

Let $M$ be a compact, connected, smooth manifold, let $g$ be a smooth 
Riemannian metric tensor,
and let $\mu= \mu_g$ denote the associated volume measure. 
For each Borel set $E$, let $\vol(E):=\mu(E)$.
 Let $C^{\infty}(M)$ denote the space of real- or
complex-valued smooth functions on $M$. For $u, v \in C^{\infty}(M)$, 
define the $L^2$ inner product 
\[  \iprod{ u,v } ~:=~ \int_{M} u \cdot \overline{v} \, d\mu \]
and $L^2$ norm  $\|u\| := \sqrt{\iprod{ u,u}}$. Then $L^2(M)=L^2(M, \mu)$
is the completion of $C^{\infty}(M)$ with respect to this norm.  

The Riemannian gradient $\nabla= \nabla_g$ is defined implicitly by 
$g(\nabla u, X)=Xu$  for each $u \in C^\infty(M)$ and vector field $X$ on $M$. 
The Laplacian $\Delta = \Delta_g$ is the unique 
self-adjoint, closed 
operator\footnote{ 
The extension is given by the Friedrichs extension. See, for example,
Theorem X.23 in 
\cite{Reed-Simon-II}.}
on $L^2(M)$ such that for each $u, v \in C^{\infty}(M)$,
\begin{equation}   \label{eq:dirichlet}
\iprod{ \cL u, v}
  ~=~
\int g(\nabla u, \nabla \overline{v}) \,d\mu
  \,.
\end{equation}
Since $M$ is compact and $\Delta$ is elliptic, standard results 
give that $(\Delta+I)^{-1}$ is a compact operator. It follows 
that the spectrum of $\Delta$ consists of a discrete set of
eigenvalues, and each eigenspace has a finite basis of eigenfunctions. 
Standard elliptic-regularity results imply that each eigenfunction is smooth.
Since the operator is self-adjoint, the eigenspaces are mutually orthogonal,
the eigenvalues are real, and there exists an orthonormal basis 
of real-valued eigenfunctions. By (\ref{eq:dirichlet}), the 
operator is nonnegative and hence the eigenvalues are nonnegative.  

Recall that in the introduction we defined $e_{\lambda}$ to be 
the integral kernel for the projection $E[0,\lambda]$ onto 
the space $\Hcal_{\lambda} \subset L^2(M)$. 
%It will prove convenient to 
%define notation for the restriction of $e_{\lambda}$ to $\{x\} \times M$. 
We define $F^{\lambda} \colon M \to L^2(M)$ by $x \mapsto F^\lambda_x$,
where
\[   F^{\lambda}_x(y) ~:=~  e_{\lambda}(x,y)\,. \]
Of course, the image of $F^\lambda$ is contained in $\Hcal_\lambda$, and by \eqref{defn:integral-kernel},
if given an orthonormal basis $\Bcal_\lambda$ for $\Hcal_\lambda$, 
then the coordinates of $F_x^\lambda$ are 
$\Seq{b(x) \st b \in \Bcal_\lambda}$.

In some contexts, the function $F^{\lambda}$ 
is called a \dfn{spectral embedding} \cite{Shayan}.\footnote{We do not know for which $\lambda$
the map $x \mapsto F_x^\lambda$ is a topological embedding.} 
In what follows, we will usually 
suppress from notation the dependence of $F^{\lambda}_x$
on $\lambda$.

Because $(x,y) \mapsto F_x(y)$ is the kernel for the projection onto $\Hcal_{\lambda}$
and $F_x \in \Hcal_{\lambda}$, we have 
\begin{equation}  \label{eqn:kernel-identity}
  F_x(y) ~=~ \int_{M} F_y(z) \cdot F_x(z)\, d\mu(z)\,.
\end{equation}
In particular, 
\begin{equation}  \label{eqn:normF}
    F_x(x) ~=~  \norm{F_x}^2,
\end{equation}
and by the Cauchy--Schwarz inequality,
\begin{equation} \label{est:F}
  |F_x(y)| ~\leq~ \|F_y\| \cdot \|F_x\| ~=~ \sqrt{F_y(y) \cdot F_x(x)} \,.
\end{equation}

\iffalse
For each $\alpha\in\, (0, 1)$, define 
$\Omega_\alpha(x): =\big\{y\st |F_x(y)|\, \ge\, \alpha \cdot F_x(x)\big\}$. 

\begin{lem}  \label{lem:ctvol}
%
For each $\lambda \in \Rbb$ and $0 <\alpha < 1$, we have
%
\[  
\ct(\lambda)~ 
\le~
\frac{1}{\alpha^2}
\int_M \frac{1}{\vol\Omega_\alpha(x)} d\mu(x)
\,.
\] 
\end{lem}

\begin{proof}
%
As indicated in the introduction, a straightforward calculation using (\ref{defn:integral-kernel})
gives
%
\begin{equation}  \label{eq:count}
%
 \ct(\lambda)~  =~  \int_{M \times M} e_{\lambda}(x,y)^2\, d (\mu \times \mu)(x, y)~ 
%
     =~  \int_{M} \|F_x\|^2\, d \mu(x).
%
\end{equation}
% 
We have
\[
\norm{F_x}^2~
=~
\int_M F_x(y)^2 \,d\mu(y)~
\ge~
\int_{\Omega_\alpha(x)} F_x(y)^2 \,d\mu(y)~
\ge~
\alpha^2 \cdot \norm{F_x}^4 \cdot \vol\Omega_\alpha(x)\,
\]
%
where in the last inequality we have used (\ref{eqn:normF}).
The result then follows by dividing both sides by 
$\alpha^2 \norm{F_x}^2 \vol\Omega_\alpha(x)$, integrating over $M$,
and applying (\ref{eq:count}). 
%
\end{proof}
\fi

By definition, $\nabla F_x$ is the gradient of $e_\lambda(x,y)$
in the second variable. We will write $\nabla_x F_x$ to indicate
the gradient of $e_\lambda(x,y)$ in the first variable.  
For each vector $X$ in the tangent bundle of $M$,
let $|X| := \sqrt{g(X, X)}$. 

\begin{lem}
\label{lem:intgradbound}
For every $\lambda \ge 0$, we have
$$
\int_M \norm{|\nabla_x F_x|}^2 \,d\mu(x)~
\le~
\lambda \int_M \norm{F_x}^2 \,d\mu(x)
\,.
$$ 
\end{lem}

\begin{proof}
Let $\Bcal_{\lambda}$ be a real orthonormal basis for $\Hcal_{\lambda}$.
An elementary calculation shows that
\begin{align*}
\norm{|\nabla_x F_x|}^2
&~=~
\int_M |\nabla_x F_x(y)|^2 \,d\mu(y)
\\ &~=~
\int_M
\mg\Bigl(\sum_{b \in \Bcal_\lambda} \nabla b(x) {b(y)},
\sum_{\tb \in \Bcal_\lambda} \nabla \tb(x) {\tb(y)}\Bigr) \,d\mu(y)
\\ &~=~
\sum_{b, \tb \in \Bcal_\lambda} \mg\bigl(\nabla b(x), \nabla
\tb(x)\bigr) \cdot
\iprod{b, \tb}
~=~
\sum_{b \in \Bcal_\lambda} |\nabla b(x)|^2
\,.
\end{align*}
Now integrate over $M$ and use \eqref{eq:dirichlet} to see that for each $b
\in \Bcal_\lambda$, there is $\lambda' \le \lambda$ such that $\int_M
|\nabla b(x)|^2 \,d\mu(x) = \lambda' \norm{b}^2 \le \lambda
\norm{b}^2$; summing gives the result.
\end{proof}

\iffalse
\begin{proof}
First, an elementary calculation shows that
\begin{align*}
\norm{|\nabla_x F_x|}^2
&=
\int_M |\nabla_x F_x(y)|^2 \,d\mu(y)
=
\int_M
\mg\Bigl(\sum_{\lambda_i \le \lambda} \nabla \ef_i(x) \cef_i(y),
\sum_{\lambda_j \le \lambda} \nabla \ef_j(x) \cef_j(y)\Bigr) \,d\mu(y)
\\ &=
\sum_{\lambda_i, \lambda_j \le \lambda} \mg\bigl(\nabla \ef_i(x), \nabla
\ef_j(x)\bigr)
\iprod{\ef_j, \ef_i}
=
\sum_{\lambda_i \le \lambda} |\nabla \ef_i(x)|^2
\,.
\end{align*}
Now integrate over $M$ and use \eqref{eq:dirichlet} to see that $\int_M
|\nabla \ef_i(x)|^2 \,d\mu(x) = \lambda_i \norm{\ef_i}^2 \le \lambda
\norm{\ef_i}^2$; summing over $i$ gives the result.
\end{proof}
\fi

%%%%%%%%%%%%%%%%%%%%%%%%%%%%%%%%%%%%%%%%%%%%%%%%%%%%%%%%%%%%%%%%%%%%%%%%%%%%%%%%

\section{The upper bound in the case of compact homogeneous spaces}
\label{sec:transitive}

In this section we assume that $(M,g)$ is a compact Riemannian \dfn{homogeneous space},\footnote{Each homogeneous space is smooth, and, in fact, analytic. This follows from the solution of Hilbert's Fifth problem. See, for example, \S 2.1 in \cite{Va}.}
in other words, that for each pair $(x,y) \in M \times M$, there exists a 
self-isometry $\iota$ of $(M,g)$ so that $\iota(x)=y$. 
Because the Laplacian commutes with (pull-back by) each isometry, 
so does each spectral projection. Also, the gradient commutes with isometries.
It follows that $\norm{F_x}$ and $\norm{|\nabla_x F_x|}$ do not 
depend on $x \in M$.
In particular, Lemma \ref{lem:intgradbound} yields
\begin{equation}
\label{eq:gradbound}
\norm{|\nabla_x F_x|}^2~ 
\le~
\lambda \norm{F_x}^2 
\end{equation}
for every $x \in M$.

Let $B(x,r) \subset M$ denote those points whose Riemannian distance to $x$
is less than $r$. Since $(M,g)$ is homogeneous, the volume, $V(r)$, of $B(x,r)$ 
does not depend on $x \in M$.

\iffalse
\begin{prop} \label{prop:ball-in-Omega}
%
If $(M,g)$ is a connected, compact, Riemannian homogeneous space,
and  if  $0<\theta < \pi/2$ and  $\lambda > 0$ then
%
\[ B(x, \lambda^{-\frac{1}{2}} \cdot \theta)~ \subset~ \Omega_{\cos \theta}. \]
%
\end{prop}

\begin{proof}
%
Since $(M,g)$ is homogeneous,  $x \mapsto \norm{F_x}$ is a constant function.
In particular, there exists $\rho>0$ such that $F$ maps $M$ into the sphere 
of radius $\rho$ in $L^2(M)$. 
Given $x, y \in M$, choose a path $\gamma$ from $x$ to $y$ of length
$\dist(x, y)$. Using  \eqref{eq:gradbound} and the chain rule, we find 
that $F \circ \gamma$ has length at most $\sqrt\lambda \cdot \rho \cdot \dist(x, y)$. 
The distance between two points in the sphere of radius $\rho$
equals $\rho \cdot \psi$, where $\psi$ is the angle between the two points.
In particular, $\rho \cdot \psi = \dist(F_x, F_y)$ where $\cos(\psi) = \iprod{F_x, F_y}/\rho^2$.
Thus, by using (\ref{eqn:kernel-identity}), (\ref{eqn:normF}), 
and the fact that $\dist(F_x,F_y)$ is majorized by the length of $F \circ \gamma$, 
we find that
%
\begin{equation} 
\label{eq:lip}
%
  \arccos\, \frac{F_x(y)}{F_x(x)}~ 
%
   =~ \arccos \frac{ \iprod{F_x, F_y}}{\rho^2}
%
  \leq~ \sqrt \lambda \cdot \dist(x, y).
%
\end{equation}
%
In particular, if $\sqrt\lambda \cdot \dist(x,y) \leq  \theta < \pi/2$,
then $F_x(y) \ge F_x(x) \cdot \cos \theta$. 
%
\end{proof}

By combining Proposition \ref{prop:ball-in-Omega} with Lemma \ref{lem:ctvol},
we obtain our main result. 
\fi

\begin{thm}
\label{thm:transitiveeig}
If $(M,g)$ is a connected, compact, Riemannian homogeneous space, 
then for each $\alpha\in (0,\pi/2)$ and each $\lambda > 0$, 
\begin{equation}
\label{eq:transitiveeig}
\ct(\lambda)~ \le~
\frac{\vol(M)}{\int_0^{\pi/2} V(\theta/\sqrt \lambda) \sin(2\theta)
\,d\theta}
~\leq~ \frac{\vol(M)}{\cos^2 \alpha\cdot
V(\lambda^{-1/2} \cdot \alpha)}\,.
\end{equation}
\end{thm}

If, for example, $(M,g)$ is the circle $\Rbb/\Zbb$, then 
the first inequality in (\ref{eq:transitiveeig}) 
gives $\ct(\lambda) \le \lfloor 2 \sqrt \lambda/\pi \rfloor$ for $\lambda \ge \pi^2$.
For comparison, the exact value of $\ct(\lambda)$ is 
$2\lfloor\sqrt\lambda/(2\pi) \rfloor + 1$ for $\lambda \ge 0$.

\begin{proof}
Fix $\lambda > 0$.
Since $(M,g)$ is homogeneous,  $x \mapsto \norm{F_x}$ is a constant function.
That is, there exists $\rho>0$ such that $F$ maps $M$ into the sphere 
of radius $\rho$ in $L^2(M)$. Thus, since  $\ct(\lambda)= \rho^2 \cdot \vol(M)$ 
as described in \eqref{eq:int-spec}, it suffices to bound $\rho^{-2}$ from below.

Given $x, y \in M$, choose a path $\gamma$ from $x$ to $y$ of length
$\dist(x, y)$. Using  \eqref{eq:gradbound} and the chain rule, we find 
that $F \circ \gamma$ has 
length at most $\sqrt\lambda \cdot \rho \cdot \dist(x, y)$. 
The distance between two points in the sphere of radius $\rho$
equals $\rho \cdot \psi$, where $\psi$ is the angle between the two points.
In particular, if $\psi(x,y)$ is the angle between $F_x$ and $F_y$,
then $\psi(x,y) \leq \sqrt{\lambda} \cdot \dist(x,y)$. 
Therefore, if $\dist(x, y) \le r_\lambda := \pi/(2 \sqrt{\lambda})$, then 
\begin{equation}
\label{est:cosine-angle}
\cos \psi(x, y)~ 
\ge~ 
\cos \left(\sqrt \lambda \cdot \dist(x,y)\right)~
\ge~
0\,.
\end{equation}

On the other hand, $\cos\, \psi(x,y)= \langle F_x, F_y \rangle/ \rho^2$.
Therefore, using (\ref{eqn:kernel-identity}) we find that 
\begin{equation}
\label{eq:value-rho}
\rho^2~ 
=~
 \norm{F_x}^2~ 
=~ 
\int_M \iprod{F_x,F_y}^2 \,d\mu(y)~ 
=~ \rho^4 \int_M \cos^2 \psi(x, y) \,d\mu(y)
\,.
\end{equation}
By combining this with 
(\ref{est:cosine-angle}), 
we find that
\begin{eqnarray}\label{eq:lowerbound}
\nonumber
\rho^{-2}
&\ge&
\int_{y \in B(x, r_\lambda)} \cos^2 \bigl(\sqrt \lambda \cdot \dist(x,
y)\bigr) \,d\mu(y) 
\\
&=&
 \int_0^{r_\lambda}\, |\partial B(x, r)|
\cdot \cos^2 (\sqrt \lambda r)
\,dr
\,,
\end{eqnarray}
where the latter equality is due to the coarea formula and $|\,\cdot\,|$ denotes hypersurface area.
A change of variable and
integration by parts show that the right-hand side
equals
$\int_0^{\pi/2} V(\theta/\sqrt \lambda) \cdot \sin(2\theta) \,d\theta$. The first inequality in  \eqref{eq:transitiveeig} follows.

The second inequality of \eqref{eq:transitiveeig} is obtained by
integrating in \eqref{eq:lowerbound} only from $r = 0$ to
$r = \alpha/\sqrt\lambda$ and bounding $\cos^2 (\sqrt\lambda r)$ from 
below by $\cos^2 \alpha$.
\end{proof}

\begin{remk} \label{remk:subspace}
If one were to replace $\Hcal_{\lambda}$ with an isometry-invariant
subspace $H \subseteq \Hcal_\lambda$
and $F^{\lambda}$ with the spectral embedding $F^H$ associated 
to the projection onto $H$, then all of the above analysis would apply
to $F_x^H$ with $\ct(\lambda)$ replaced by 
\[  \dim(H) ~=~ \int_M \left\|F_x^H\right\|^2\, d\mu(x)\,. \]  
\end{remk}

\begin{remk} \label{remk:Libound}
The estimate of the first nonzero eigenvalue in \cite{Li} described in (\ref{eq:transitivegap})
%
%\begin{equation} \label{remk:Li-revisited}
%
% \sqrt{\lambda_1} \geq \frac{\pi}{2 \sup_{x,y} \dist(x,y)} 
%
%\end{equation}
%
does not follow directly from Theorem \ref{thm:transitiveeig}. 
However, this estimate follows easily from the analysis 
leading to Theorem \ref{thm:transitiveeig}.
Indeed, let $H$ be the $\lambda_1$-eigenspace.
%Then by Remark \ref{remk:subspace}, the above analysis 
%applies to $F^H_{x}$. 
We have $\cos \psi^H(x,y)= \langle F^{H}_x, F^H_y \rangle / \rho^2$, where 
$\psi^H(x,y)$ is the angle between $F^{H}_x$ and  $F^H_y$.
Since $F^H_x$ is orthogonal
to the constants, there exists $y$ so that $F_x^H(y) < 0$.   
Therefore $\dist(x,y) \cdot \sqrt{\lambda} \ge \psi^H(x, y) > \pi/2$
and (\ref{eq:transitivegap}) follows.  
\end{remk}

We now improve Li's bound.

\begin{thm} \label{thm:Li-improved}
If $(M,g)$ is a connected, compact, Riemannian homogeneous space, then
\begin{equation} \label{est:Li-improved}
\sqrt{\lambda_1}~
>~
\frac{\pi}{2D}~ +~
\frac{1}{D}
\cdot
\arcsin \left( \frac{V(D/2)}{2 \bigl(\vol M - V(D/2)\bigr)} \right)\,,
\end{equation}
where $D$ is the diameter of $(M,g)$.
\end{thm}

\begin{proof}
We first remark that (\ref{est:Li-improved}) is true if 
$\sqrt{\lambda_1} \cdot D > 2\pi/3$.
Indeed, if $\dist(x,y)=D$, then the open balls of radius $D/2$ 
centered respectively at $x$ 
and $y$ are disjoint, and hence $2 \cdot V(D/2)\leq \vol(M)$. 
Thus, $V(D/2)\big/\bigl(\vol(M)-V(D/2)\bigr) \leq 1$, 
and so the expression 
on the right side of (\ref{est:Li-improved}) is bounded 
above by $2 \pi/3D$.
Thus, we assume for the remainder of the proof that
$\sqrt{\lambda_1} \cdot D \le 2\pi/3$.

Let $H$ and $F^H$ be as in Remark \ref{remk:Libound}.
Since $F^H_x$ is orthogonal to the constants, we find from 
(\ref{est:cosine-angle}) and the coarea formula that
\begin{equation} 
\label{eq:intcos}
0~
=~
\int_M \frac{F^H_x(y)}{F^H_x(x)} \,d\mu(y)~
\ge~
\int_0^D |\partial B(x, r)|\, \cos(\sqrt\lambda_1 r) \,dr \, .
\end{equation}

Let $f(r) := \cos(\sqrt{\lambda_1} \cdot r)$. Now
$f(r)$ is a decreasing function
on $[0,D]$. In particular, from (\ref{eq:intcos}) we see that
$f$ has a unique zero $r_0 := \pi/2 \sqrt{\lambda_1}$ on $[0,D]$.
Hence we find from (\ref{eq:intcos}) that
\[   
\int_{r_0}^{D} \left|\partial B(0,r) \right| \cdot |f(r)|\, dr~
\geq~ 
\int_{0}^{\frac{2}{3} \cdot r_0} 
|\partial B(0,r)| \cdot f(r) \, dr 
\,.   
\]
Therefore, since the minimum value of $f$ over $[0, \frac{2}{3} r_0]$ is 
$1/2$ and the maximum value of $|f|$ over $[r_0, D]$ is 
$|\!\cos(\sqrt{\lambda_1} D)|$, we have 
\begin{equation}
\label{est:post-split}  
\left|\cos\bigl(\sqrt{\lambda_1} D \bigr) \right| \cdot \bigl( \vol(M)- V(r_0) \bigr)~
>~  \frac{1}{2}  \cdot V \left( \frac{2}{3} \cdot r_0 \right) 
\,.  
\end{equation}
Since, by assumption, $\pi/2 < \sqrt{\lambda_1} \cdot D \leq \pi$, we have 
$\left|\cos\bigl(\sqrt{\lambda_1} D \bigr) \right| = 
\sin(\sqrt{\lambda_1} D - \pi/2)$. Since
we have assumed that 
$\frac{2}{3} r_0 \le D/2$, it follows from (\ref{est:post-split}) that
\[
\sin\left(\sqrt{\lambda_1} \cdot D\, -\, \frac{\pi}{2} \right)~
\geq~ 
\frac{1}{2}  \cdot
\frac{V \left( D/2 \right)}{\bigl( \vol(M)- V(D/2) \bigr)}\, ,
\]
and the claim follows. 
\iffalse
If $\sqrt{\lambda_1} > 2\pi/(3D)$, then this inequality is easily verified,
so suppose that $\sqrt{\lambda_1} \le 2\pi/(3D)$.
Then from (\ref{eq:lip}) but in the notation of Remark
\ref{remk:Libound}, we have 
$F^H_x(y)/F^H_x(x) = \cos \psi(x, y) \ge \cos \bigl(\sqrt \lambda_1
\dist(x, y)\bigr)$ for all $x, y$,
whence
\begin{equation}  %\label{eq:intcos}
0~
=~
\int_M \frac{F^H_x(y)}{F^H_x(x)} \,d\mu(y)~
\ge~
\int_0^D |\partial B(x, r)|\, \cos(\sqrt\lambda_1 r) \,dr
%
\end{equation}
%
by the coarea formula. 

Now $\cos(\sqrt\lambda_1 r) \ge 0$ for $0 \le r \le
\pi/(2\sqrt{\lambda_1})$ and 
$\cos(\sqrt\lambda_1 r) \le 0$ for $\pi/(2\sqrt{\lambda_1}) \le r \le D$.
Since
%
\[
\int_0^{\pi/(3\sqrt{\lambda_1})} |\partial B(x, r)|\, \cos(\sqrt\lambda_1 r) \,dr
>
V\bigl(\pi/(3\sqrt{\lambda_1})\bigr)/2
\]
%
and
\begin{align*}
\int_{\pi/(2\sqrt{\lambda_1})}^D |\partial B(x, r)| \cdot |\!\cos(\sqrt\lambda_1
r)| \,dr
&<
\int_{\pi/(2\sqrt{\lambda_1})}^D |\partial B(x, r)| \cdot |\!\cos(\sqrt\lambda_1
D)| \,dr
\\ &<
|\!\cos(\sqrt\lambda_1 D)|\, \bigl(\vol M - V\bigl(\pi/(3\sqrt{\lambda_1})\bigr)\bigr)
\,,
\end{align*}
it follows that
\[
|\!\cos(\sqrt\lambda_1 D)|~
>~
\frac{V\bigl(\pi/(3\sqrt{\lambda_1})\bigr)}{2\bigl(\vol M -
V\bigl(\pi/(3\sqrt{\lambda_1})\bigr)\bigr)}~
\ge~
\frac{V(D/2)}{2\bigl(\vol M - V(D/2)\bigr)}
\,.
\]
Since $|\!\cos(\sqrt\lambda_1 D)| = \sin \bigl(\sqrt{\lambda_1} D -
\pi/2\bigr)$, the claim follows.
\fi
\end{proof}

\begin{remk} \label{remk:antipode}
This method also shows immediately that for
any round sphere,
we have
$\sqrt{\lambda_1} \ge \pi/D$:
If $\sqrt{\lambda_1} \le \pi/D$,
then the inequality \eqref{eq:intcos}
holds. But since 
$|\partial B(x, r)| = |\partial B\bigl(y, D - r\bigr)| = |\partial B(x, D - r)|$ when $\dist(x, y) = D$, this is easily seen to be impossible if
$\sqrt{\lambda_1} < \pi/D$.
In the case of the circle, this is sharp since
$\sqrt{\lambda_1} = \pi/D$.
\end{remk}

%%%%%%%%%%%%%%%%%%%%%%%%%%%%%%%%%%%%%%%%%%%%%%%%%%%%%%%%%%%%%%%%%%%%%%%%%%%%%%

\section{General homogeneous spaces} \label{sec:noncompact}

In this section, we extend our analysis to 
noncompact manifolds, $M$. 
As before, the Laplacian associated to a Riemannian metric $g$
on $M$ is defined implicitly by 
\begin{equation}   \label{eq:dirichlet2}
\iprod{ \cL u, v}
  ~=~
\int_M g(\nabla u, \nabla \overline{v}) \,d\mu
\end{equation}
where, at first, the functions $u,v$ are smooth and compactly supported. 
If the metric $g$ is complete, then $\Delta$ has a unique extension to an 
unbounded self-adjoint operator on $\Hcal = L^2(M, \mu_g)$, which we continue 
to denote by $\Delta$.\footnote{See for example \cite{Gaffney} or \cite{Strichartz}.} The domain, $\dom(\Delta)$, 
of this operator can be characterized using the method of 
Friedrichs.\footnote{See, for example, Theorem  X.23 in \cite{Reed-Simon-II}.}
In particular, the sesquilinear form 
$q(u,v):= \langle \Delta u, v\rangle + \langle u, v\rangle$ 
on $C^{\infty}_0(M)$ has a unique closed extension, still denoted 
$q$, with domain $\qdom$. Let $\qdom'$ denote the linear functionals 
on $\Hcal$ that are bounded with respect to the norm $\| \cdot \|_q$ 
associated to $q$. Since $\|\cdot \| \leq \|\cdot\|_q$, we have an 
induced  sequence of embeddings 
\[  \qdom \stackrel{i}{\longrightarrow} \Hcal \stackrel{j}{\longrightarrow}  \qdom'. \]
The Riesz representation theorem defines an isometric isomorphism 
$B\colon \qdom \to \qdom'$ such that $(Bu)(v)=q(v,\overline u)$ for each $v \in \qdom$. 
The domain $\dom(\Delta)$ is
by definition equal to $i\bigl(B^{-1}(j(\Hcal))\bigr)$ and 
$\Delta := j^{-1}\circ B \circ i^{-1} - I$.

In general, the spectrum of $\Delta$ is not discrete.
Since $\Delta$ is self-adjoint and nonnegative, 
the spectral theorem provides a projection-valued measure $E$ on $\Rbb$ 
such that $\Delta = \int_{0}^{\infty} \nu \, dE_{\nu}$. Let $E_{\lambda} :=
E\bigl([0, \lambda]\bigr)$ and denote by $\Hcal_\lambda$ the range of $E_\lambda$.
The space $\Hcal_{\lambda}$
consists of smooth functions, and the operator $E_\lambda$ has a smooth, 
symmetric integral kernel $e_\lambda$ that is sometimes
called the \dfn{spectral function}. (See Theorem \ref{lem:kernel} 
in the Appendix.)
As before, we will set $F^\lambda_x(y) := e_{\lambda}(x,y)$, often 
suppress the superscript $\lambda$, and regard
$x \mapsto F_x$ as a mapping from $M$ to $L^2(M)$.

Using the spectral function, one defines
a local version of the spectral counting function. 

\begin{defn} 

For $\lambda>0$, define 
\begin{equation} \label{defn:count-noncompact}
 \ct_x(\lambda)~ :=~  \left\| F^{\lambda}_x\right\|^2~ =~ e_\lambda(x, x)\,. 
\end{equation}
\end{defn}

\begin{remk}
If $M$ is compact, then $\int_M \ct_x(\lambda)\, d\mu(x) = \ct(\lambda)$.  
\end{remk}

The remainder of this section is devoted to generalizing 
Theorem \ref{thm:transitiveeig}. Our first goal 
is the generalization of estimate (\ref{eq:gradbound}). 
We begin with the following facts concerning a symmetry of the spectral function.

\begin{prop} \label{prop:F-properties}
If $f \colon [0, \infty) \to \R$ is a Borel, locally bounded
function, then  
for each $(x, y) \in M \times M$ and $\lambda >0$, 
we have $f(\Delta)F_x \in \Hcal_{\lambda}$ and $(f(\Delta) F_x)(y) = (f(\Delta) F_y)(x)$. 
\end{prop}

\begin{proof}
If $h \in \Hcal_\lambda$, then $h(y) = \iprod{h, F_y}$ and
$f(\Delta)h = \int_0^\lambda f(\nu) \,dE_\nu (h) \in \Hcal_\lambda$. In
particular, since $F_x \in \Hcal_\lambda$,  we have
$f(\Delta) F_x \in \Hcal_{\lambda}$ and
\[
(f(\Delta) F_x)(y) 
~=~
\iprod{f(\Delta) F_x, F_y}
~=~
\iprod{F_x, f(\Delta) F_y}
~=~
(f(\Delta) F_y)(x) 
\]
since $f(\Delta)$ is symmetric.
%
%By the way, $E_\nu F^\lambda_x = F^\nu_x$.
%
\end{proof}

\begin{coro} \label{coro:Lap-integrable}
For each $(x,y) \in M \times M$, we have $\Delta_x F_x(y) = \Delta_y F_x(y)$.
In particular, $y \mapsto \Delta_x F_x(y)$ is square-integrable with respect to $d \mu(y)$.
\end{coro}
\begin{proof}
Because the kernel $e_{\lambda}$ is symmetric, 
we have $\Delta_x F_x(y) = \Delta_x e_{\lambda}(x,y) \allowbreak = \Delta_x e_{\lambda}(y,x)\allowbreak = \Delta_x F_y(x)$.
By Proposition \ref{prop:F-properties}, we have $\Delta_x F_y(x)= \Delta_y F_x(y)$,
and $\Delta F_x$ is square-integrable. 
\end{proof}

\begin{lem} \label{lem:e-invariance}
Let $(M,g)$ be a complete Riemannian manifold.
If $\iota\colon M \to M$ is an isometry of the Riemannian metric $g$,
then  for each $(x,y) \in M \times M$, we have 
$e_{\lambda}\bigl(\iota(x), \iota(y)\bigr)= e_{\lambda}(x,y)$. 
\end{lem}

\begin{proof}
Let $\iota^*\colon \Hcal \to \Hcal$ denote the operator defined by $\bigl(\iota^*(u)\bigr)(x) = u\bigl(\iota(x)\bigr)$. 
Since $\iota$ is an isometry, $\iota^*$ commutes with $\Delta$, and hence it commutes
with $E_{\lambda}$ for each $\lambda \in \Rbb$.   Thus, for each 
smooth, compactly supported $u\colon M \to \Rbb$, we have 
\begin{eqnarray*}
\int_M u(y) e\bigl(\iota(x), y\bigr)\, d\mu(y)  & = &  \langle u, F_{\iota(x)} \rangle \\[-4pt]
& = &  \iota^*(E_{\lambda}(u))(x) \\
& = &  E_{\lambda}(\iota^*(u))(x) \\
& = &  \langle \iota^*(u) , F_x \rangle \\
& = &  \int_M u\bigl(\iota(y)\bigr) \cdot e(x,y)\, d \mu(y)  \\
& = &  \int_M u(z) \cdot e \left(x,\iota^{-1}(z) \right) d \mu(z)\,, \\
\end{eqnarray*}
where in the the last equality we used the fact that isometries
preserve the Riemannian measure $\mu$. 
Since $u$ is an arbitrary compactly supported function, 
we have $e\bigl(\iota(x), y\bigr)= e\bigl(x, \iota^{-1}(y)\bigr)$, and the claim follows.
\end{proof}

Since each homogeneous space is complete, we deduce the following. 

\begin{coro} \label{coro:diagonal-constant}
If $(M,g)$ is a connected, Riemannian homogeneous space, 
then $x \mapsto F_x(x)=e_{\lambda}(x,x)$ is constant.\qed
\end{coro}

The next lemma asserts, in part, that for each $x \in M$, the function 
$y \mapsto |\nabla_x F_x (y)|^2$ is integrable with respect to $d \mu(y)$. Note that the more common statement would be that $y \mapsto |\nabla_y F_x (y)|^2$ is integrable.

\begin{lem} \label{lem:Laplacian-gradient}
If $(M,g)$ is a connected, Riemannian homogeneous space, 
then for each $x \in M$, the function $y \mapsto |\nabla_x F_x(y)|$
belongs to $\Hcal$, and we have 
\begin{equation} \label{eq:int-by=parts-in-y}
   \int_M | \nabla_x F_x(y)|^2\, d\mu(y)~
 =~ \int_M F_x(y) \cdot \Delta_x F_x(y)\, d \mu(y)\,. 
\end{equation}
\end{lem}

\begin{proof}
A well-known formula gives 
\[ \Delta_x F_x^2~ =~ 2\, F_x \cdot \Delta_x F_x~ -~ 2\, |\nabla_x F_x|^2. \]
By Proposition \ref{prop:gradient-estimate} in the Appendix, the 
function $y \mapsto |\nabla_x F_x |^2$ is integrable with respect to $d \mu(y)$, 
and by Corollary \ref{coro:Lap-integrable},
$y \mapsto F_x(y) \cdot \Delta_x F_x(y)$ is integrable.
Hence $y \mapsto \Delta_x F_x^2(y)$ is integrable with integral not depending on $x$ by homogeneity, and therefore, 
to verify (\ref{eq:int-by=parts-in-y}), it suffices to prove
that $\int \Delta_x F_x^2(y)\, d\mu(y) =0$. 

Let $u\colon M \rightarrow \Rbb$ be a compactly supported smooth function. 
Using Fubini's theorem, the fact that $\Delta$ is self-adjoint,
and the identity $\|F_x\|^2 = F_x(x)$,  we find that
\begin{eqnarray*}
  \int_M u(x) \left(\int_M \Delta_x F_x^2(y)\, d \mu(y) \right)\, d\mu(x)  &=&
  \int_{M} \int_{M} u(x) \cdot \Delta_x F_x^2(y)\, d \mu(x)\, d \mu(y) \\
  &=& \int_M \int_M \left(\Delta u(x) \right) \cdot F_x(y)^2\, d \mu(x)\, d\mu(y) \\
  &=&  \int_M \Delta u(x) \int_M  F_x(y)^2 \, d\mu(y)\, d\mu(x) \\
  &=&  \int_M (\Delta u(x)) \cdot F_x(x) \, d\mu(x) \\
  &=& \int_M  u(x) \cdot \Delta_x F_x(x)\, d \mu(x)\,.
\end{eqnarray*}
Corollary \ref{coro:diagonal-constant} implies that the last integral equals zero.
Since $u$ is arbitrary, the claim follows. 
\end{proof}

\begin{prop}  \label{coro:gradient-est-noncompact}
If $(M,g)$ is a connected, Riemannian homogeneous space, 
then for each $x \in M$, we have 
\[   \left\| | \nabla_x F_x| \right\|~ \leq~ \sqrt{\lambda} \cdot \left\|F_x\right\|. \]
\end{prop}

\begin{proof}
Since $F_x = E_{\lambda}(F_x)$, 
we have
\begin{eqnarray*} 
\langle F_x, \Delta F_x\rangle
 &=& \left\langle \int_0^{\lambda} dE_{\nu}(F_x),  \int_0^{\lambda} \nu \cdot dE_{\nu}(F_x) \right\rangle \\
      &=& \int_0^{\lambda} \nu \cdot  d\|E_{\nu}(F_x)\|^2 \\
      &\leq&  \lambda  \int_0^{\lambda}   d\|E_{\nu}(F_x)\|^2 \\
       &=& \lambda \cdot \|F_x\|^2.
\end{eqnarray*} 
The claim then follows from  Lemma \ref{lem:Laplacian-gradient}.
\end{proof}

We now give the general version of Theorem \ref{thm:transitiveeig}. 
Note that if $(M,g)$ is homogeneous, then $x \mapsto \ct_x(\lambda)$ 
is constant.

\begin{thm}  
\label{thm:homogeneous-noncompact}
If $(M,g)$ is a connected, Riemannian homogeneous space, then 
for each $x \in M$,
$\>\alpha\in\, (0,\pi/2)$, and $\lambda > 0$, we have
\begin{equation}
\label{eq:homogeneous-noncompact}
\ct_x(\lambda)~ \le~
\frac{1}{\int_0^{\pi/2} V(\theta/\sqrt \lambda) \sin(2\theta)
\,d\theta}
~\leq~ \frac{1}{\cos^2 \alpha\cdot
V(\lambda^{-1/2} \cdot \alpha)}\,.
\end{equation}
\end{thm}

\begin{proof}
The proof is the same as the proof of Theorem \ref{thm:transitiveeig} 
except that one replaces inequality
(\ref{eq:gradbound}) with Proposition
\ref{coro:gradient-est-noncompact}.
\end{proof}

%%%%%%%%%%%%%%%%%%%%%%%%%%%%%%%%%%

\section{An upper bound for the heat kernel} \label{sec:spectrals}

Let $(M,g)$ be a homogeneous Riemannian manifold, and let $\Delta$ be 
the associated Laplacian. 
The heat kernel $p_t(x, y)$ is the integral kernel for the heat operator, $e^{-t \Delta}$.
The heat kernel has two important interpretations:
It is the integral kernel for the solution operator for the heat equation 
$\partial_t u(x,t)= \Delta_x u(x,t)$ \cite{Li-Book,Davies}. 
From the probabilistic viewpoint, the heat kernel is the 
transition subprobability density of Brownian motion on $(M,g)$, the minimal diffusion whose
infinitesimal generator is $\Delta$.

\cite{Pittet} shows that there are three types of noncompact homogeneous spaces:
\begin{itemize}
\item $M$ is nonamenable, $\ct_x(\lambda) = 0$ for some $\lambda > 0$, 
and there are constants $c_1, c_2 \in (0, \infty)$
such that $c_1 e^{-c_2 t} \le p_t(x, x) \le c_2 e^{-c_1 t}$ for all $t > 1$;
\item $M$ is amenable with exponential volume growth, $\ct_x(\lambda) > 0$ for all $\lambda > 0$, 
and there are constants $c_1, c_2 \in (0, \infty)$
such that $c_1 e^{-c_2 t^{1/3}} \le p_t(x, x) \le c_2 e^{-c_1 t^{1/3}}$ for all $t > 1$;
\item $M$ has polynomial volume growth of order $d$ (the dimension of $M$), 
and there are constants $c_1, c_2 \in (0, \infty)$
such that $c_1 t^{-d/2} \le p_t(x, x) \le c_2 t^{-d/2}$ for all $t > 1$.
\end{itemize}
Furthermore, these three cases can be distinguished by their corresponding isoperimetric profiles.
In the case of Lie groups, even more refined information is known \cite{Varop}.
The results of both of these authors
are difficult; we show here how to obtain comparable upper bounds on the heat
kernel that depend only on volume information, as well as to provide explicit constants
in those upper bounds. Note that exponential volume growth
can occur for both amenable and nonamenable Lie groups, so that we cannot hope to obtain exponential
decay of the heat kernel merely from a bound on the volume growth.

By Proposition \ref{lem:kernel2}, the heat kernel is
the Laplace transform of the spectral function (as a measure): 
\begin{equation}\label{eq:laplace}
p_t(x, y)~ =~ \int_0^\infty e^{-\lambda t} \,de_\lambda(x, y)\,.
\end{equation}

It is well known that bounds on the spectral function lead to bounds on
the heat kernel. 
We give two illustrations.
Let $\Gamma$ be the classical gamma function.

\begin{thm} \label{thm:polygrowth}
If $(M,g)$ is a homogeneous space such that $V(r) \geq c \cdot r^\beta$
for all $0 < r < D$, then for each $(x,y) \in M \times M$ and $t>0$,
\begin{equation} \label{est:heat-kernel}
   p_t(x,y)~ \leq~ \frac{1}{\vol(M)}~  
+~ \frac{\Gamma(\beta/2+1) \cdot 2^{\frac{\beta}{2}} }{c \cdot m_\beta} \cdot
t^{-\frac{\beta}{2}} \,,
\end{equation}
where $m_\beta := \int_0^{\pi/2} \theta^\beta \sin(2\theta) \,d\theta$.
If $\vol M = \infty$, then we interpret $1/\vol M$ as $0$.
\end{thm}

\begin{proof}
From the semigroup property of the heat kernel and the Cauchy--Schwarz inequality,
one finds that 
\[
|p_t(x,y) - 1/\vol M|~
\leq~
\bigl([p_{t}(x,x) - 1/\vol M] \cdot
[p_{t}(y,y) - 1/\vol M]\bigr)^{\frac{1}{2}}
\,.
\]
Hence, it suffices 
to prove (\ref{est:heat-kernel}) for $x = y$.
In that case, since $de_{\lambda}(x,x)= dN_x(\lambda)$, we have  
\begin{eqnarray}
  p_t(x,x)
    &=& \int_{0}^{\infty} e^{-\lambda \cdot t}\, d\ct_x(\lambda)  \label{eqn:heat-kernel} \\
   &=&   \frac1{\vol M}
   ~ +~          t \int_{0}^{\infty} e^{-\lambda \cdot t} \cdot
   \ct_x(\lambda)\, d \lambda\,.
   \nonumber
\end{eqnarray}
By Corollary
\ref{coro:volume-ass} (in the compact case) or Corollary \ref{coro:volume-ass-gen} (in general),
$\ct_x(\lambda) \leq c^{-1} \cdot m_\beta^{-1} \cdot \lambda^{\frac{\beta}{2}}$ for
all $\lambda > 0$,
and so
\[ p_t(x,x) ~ \leq~ \frac1{\vol M} ~ +~
      \frac{t}{c \cdot m_\beta}
    \int_{0}^{\infty} e^{-\lambda \cdot t} \cdot \lambda^{\frac{\beta}{2}}\, d
    \lambda\,.
\]
The claim then follows from 
change of variable in the definition of $\Gamma$.
\end{proof}

For our next illustration, we use the upper incomplete gamma function,
\[
\Gamma(a, x) ~:=~ \int_x^\infty s^{a-1} e^{-s}\,ds
\,.
\]
This function is useful because $\Gamma(a, x) \le \Gamma(a)$ for all $x$
and $\Gamma(a, x) < 2 x^{a-1} e^{-x}$ for $x > a-1$ \cite{NaPa}.
% See (1.5) in https://www.carma.newcastle.edu.au/jon/gamma4.pdf
% Uniform Bounds for the Incomplete Complementary Gamma Function
% by JONATHAN M. BORWEIN and O-YEAT CHAN
Using this function, we provide heat-kernel bounds when the volume 
of a ball of radius $r$ grows exponentially for large $r$.

\begin{thm}
Let $c_0, c_1, c_2 > 0$ be constants.
If $(M,g)$ is a homogeneous space such that $V(r) \ge c_0 \cdot r^d$ for $r \le r_0$ and
$V(r) \geq c_1 \cdot e^{c_2 r}$
for $r > r_0$, then for each $(x,y) \in M \times M$ and $t>0$,
\[
p_t(x, y)
~\le~ 
\frac{\pi c_2 + 2}{c_1}\cdot \exp\left(-c_3 \cdot t^{1/3} \right)~
+~
\frac{2}{c_0} \cdot \left(\frac{4}{\pi} \right)^d 
t^{-d/2} \cdot \Gamma\bigl(1+d/2, \pi^2t/16r_0^2\bigr)
\,,
\]
where $c_3 := (\pi c_2/4)^{2/3}$.
If in addition $\ct_x(\lambda) = 0$ for $\lambda < \lambda_*$, then
\[
p_t(x, y)
~\le~ 
\frac{2}{c_1} e^{-\lambda_* t}~
+~
 \frac{2}{c_0} \cdot \left(\frac{4}{\pi} \right)^d 
t^{-d/2} \cdot \Gamma\bigl(1+d/2, \max\{\lambda_*, \pi^2t/16r_0^2\}\bigr)
\,.
\]
\end{thm}

The first bound shows polynomial decay for small $t$ and stretched 
exponential decay for large $t$ since
the second term is less than $4^{d+1} e^{-\pi^2t/16r_0^2}/c_ \pi^d$ for 
$t > 8dr_0^2/\pi^2$.
The second bound similarly shows polynomial or exponential decay.

\begin{proof}
As in the preceding proof, it suffices to consider the case $x = y$.
Using $\alpha= \pi/4$ in Theorem \ref{thm:2ndmain}, we have 
\[
e_{\lambda}(x,x)~
\le~
\frac{2}{V(\lambda^{-\frac{1}{2}} \cdot \pi/4)}
\,,
\]
and thus using the hypothesis we find 
\[
p_t(x, x)~
\le~
2t\left(\frac{1}{c_1}\int_0^{\infty}
  \exp\left(-\lambda t-c_3^{3/2} \lambda^{-1/2}\right) \,d\lambda~
+~
\frac{4^d}{c_0 \pi^d} \int_{\frac{\pi^2}{16r_0^2}}^\infty e^{-\lambda t} \lambda^{d/2}\,d\lambda\right)
\,.
\]
We may express the second integral above as 
$t^{-1-d/2} \cdot\Gamma\bigl(1+d/2, \pi^2 t/16r_0^2\bigr)$.
For the first integral, define $\widetilde\lambda= \widetilde\lambda(t) := c_3 \cdot t^{-2/3}$.
Using 
\[
\lambda t + c_3^{3/2}\lambda^{-1/2}~
\geq~
\begin{cases}
\lambda t &\textrm{if } \lambda \ge \widetilde\lambda,\\ 
c_3^{3/2}\lambda^{-1/2} &\textrm{if } \lambda \ge \widetilde\lambda,
\end{cases}
\]
we obtain
\begin{align*}
\int_0^{\infty}  e^{-\lambda t-c_3^{3/2} \lambda^{-1/2}} \,d\lambda~
&\leq~
\int_0^{\widetilde\lambda} e^{-c_3^{3/2}\lambda^{-1/2}} \,d\lambda
+ 
\int_{\widetilde\lambda}^\infty  e^{-\lambda t} \,d\lambda
\\ &\le~ 
\int_{\widetilde\lambda^{-1/2}}^{\infty} 2s^{-3} e^{-c_3^{3/2}s} ds
+   e^{-\widetilde\lambda t}/t
\\ &\le~ 
\frac{2 c_3^{3/2} + 1 }{t} \cdot e^{-c_3 t^{1/3}}
\,.
\end{align*}
Putting together these bounds gives the first desired result. The second bound is proved
in a similar manner.
\end{proof}

\begin{remk}
For compact homogeneous $M$, we also have a bound that expresses both polynomial decay
at short times and exponential decay at large times: with the hypotheses and notation
of Theorem \ref{thm:polygrowth},
\[
p_t(x, x)
~\le~ 
\frac1{\vol M} + \frac{2^{\frac{\beta}{2}} }{c \cdot m_\beta} \cdot
t^{-\frac{\beta}{2}} \cdot \Gamma(\beta/2+1, \lambda_* t)  \,.
\]
\end{remk}

\begin{remk} \label{rem:heat-bounds-spectral}
Upper bounds on the heat kernel likewise give upper bounds on $\ct_x$.
For example, using (\ref{eqn:heat-kernel}), we find that
\[
 p_t(x, x) \geq~ e^{-\lambda \cdot t} \int_0^{\lambda}
 d\ct_x(\nu)~
  =~ e^{-\lambda \cdot t} \cdot \ct_x(\lambda)\,,
\]
and hence by setting $t:= 1/\lambda$ we obtain 
\[
\ct_x(\lambda)~ 
\le~
e\cdot p_{1/\lambda}(x, x)
\]
valid for each $\lambda>0$. For a similar result, see \cite[Proposition 5.3]{EfShu}.
Because of this general inequality, previously known bounds on the heat kernel yield
bounds on the spectral function, though with constants that may not be explicit.
\end{remk}

%%%%%%%%%%%%%%%%%%%%%%%%%%%%%%

\section{A comparison with the method of Li} \label{sec:misc}

Our geometric method requires only the $L^2$ bound \eqref{eq:gradbound} on
the gradient, rather than an $L^\infty$ bound.  Li \cite{Li} used an
$L^\infty$ bound, which, in our
language, is the following:

\begin{prop} \label{prop:Ligrad}
If $(M,g)$ is a Riemannian homogeneous space,
then for each $\lambda \geq 0$ and each $(x, y) \in M \times M$,
\begin{equation}  \label{eq:Ligrad}
|(\nabla F_x^{\lambda})(y)|^2 + \lambda \cdot F_x^{\lambda}(y)^2~
\le~
\lambda \cdot F_x^{\lambda}(x)^2
\,.
\end{equation}
\end{prop}

\cite{Li} considered only compact spaces, but this bound holds more generally.
As it might be useful in other contexts, we provide a short proof.

\begin{proof}
We follow the method of \cite[Theorem 5]{Li}. 
Since $M$ is homogeneous, there exists an isometry $\iota\colon M \to M$ 
so that $\iota(y)=x$.  By Lemma \ref{lem:e-invariance}, we have
$F_x \circ \iota^{-1}(x)=F_x(y)=  F_z(x)$ where $z:=\iota(x)$, and moreover, 
\begin{equation}
\label{eqn:diagonal-consequence}
|(\nabla F_x)(y)|~ 
=~ 
|(\nabla (F_x \circ \iota^{-1}))(x)|~ 
=~ 
|(\nabla F_z)(x)|
\, .
\end{equation}

Homogeneity and (\ref{est:F}) also imply that $F_x(y) \leq F_x(x)$
for each $y$. Therefore, $F_x$ achieves a maximum at $x$ and 
hence $(\nabla F_x)(x) = 0$. Write 
\[
F_z~ =~ a F_x\,  +\,  h \ \mbox{ with } h \perp F_x
\,.
\]
Since $h \in \Hcal_\lambda$, we have
$h(x) = \iprod{h, F_x} = 0$. Thus, $F_z(x) = a F_x(x)$, and
$(\nabla F_z)(x) = (\nabla h)(x)$. Therefore, using 
(\ref{eqn:diagonal-consequence}), we have
\begin{equation}
\label{eq:first}
|(\nabla F_x)(y)|^2 + \lambda F_x(y)^2
~=~
|(\nabla F_z)(x)|^2 + \lambda F_z(x)^2
~=~
|\nabla h(x)|^2 + \lambda a^2 F_x(x)^2
\,.
\end{equation}

Lemma \ref{lem:gradhbound} (below)
and \eqref{eq:gradbound} (or Corollary \ref{coro:gradient-est-noncompact})
together imply that
\[
|\nabla h(x)|^2~
\le~
\lambda \cdot \norm{h}^2 \cdot  \norm{F_x}^2.
\]
Since $\norm{h}^2 = \norm{F_z}^2 - a^2 \norm{F_x}^2 = (1 - a^2)
\norm{F_x}^2$, we deduce that
\[
|\nabla h(x)|^2~
\le~
\lambda \cdot (1 - a^2) \cdot \norm{F_x}^4~
=~
\lambda \cdot (1 - a^2) \cdot F_x(x)^2
\,.
\]
By combining this with \eqref{eq:first}, we obtain \eqref{eq:Ligrad}.
\end{proof}

\begin{lem}
\label{lem:gradhbound}
If (M, g) is a complete Riemannian manifold and $\lambda > 0$, then
for each $h \in \Hcal_\lambda$ and each $x \in M$, we have
$$
|(\nabla h)(x)|
~\le~
\norm{h} \cdot \norm{|\nabla_x F_x^\lambda|}
\,.
$$ 
\end{lem}

\begin{proof}
Since $h$ lies in $\Hcal_{\lambda}$, we have $h(x) = \iprod{h, F_x}$.
In the compact case, this immediately yields
\[
|\nabla h(x)|
~=~
\Big|\int h(y) \nabla_x {F_x(y)} \,d\mu(y)\Big|
~\le~
\int |h(y)|\cdot|\nabla_x F_x(y)| \,d\mu(y)
~\le~
\norm{h} \cdot \norm{|\nabla_x F_x|}
\]
by the Cauchy--Schwarz inequality.

The general case requires a more elaborate argument. 
Let $U$ be a normal coordinate neighborhood about $x$ 
with compact closure, and let 
$\partial^1, \ldots, \partial^d$ be the associated coordinate vector fields. 
By Proposition \ref{prop:gradient-estimate}, the function 
$y \mapsto |\nabla_x F_x|$ lies in $\Hcal$. 
 In addition, $x \mapsto |\nabla_x F_x(y)|$ is continuous by Theorem \ref{lem:kernel}. Fatou's lemma thus ensures that $\norm{|\nabla_x F_x|}$ is bounded for $x \in U$, whence $|\nabla_x F_x(y)|$ is square-integrable over $(x, y) \in U \times M$.
Thus, for each smooth function $u\colon M \to \Rbb$ 
with support in $U$, integration by parts gives  
\begin{multline*} 
 \int_{M} u(x) \cdot \partial^i h(x) \,d\mu(x)~
  ~=~ -\int_{M} \left( \partial^i_x u(x)\right) \cdot 
  h(x) \,  d\mu(x)\\ 
  -~ \int_M  {\rm div}(\partial^i)(x) \cdot u(x) \cdot h(x)\, d \mu(x)\,,
\end{multline*}
where ${\rm div}$ is the divergence operator associated to the Riemannian metric, $g$.
Using Fubini's theorem, integration by parts, and Fubini's theorem again, we find that the first term on the right-hand side equals
\begin{multline*}   
 -\int_{M} \int_M \left(\partial^i_x u(x) \right)
   \cdot  h(y) \cdot F_x(y) \, d \mu(y)\,  d\mu(x) \\
 = \int_{M} u(x) \int_M 
   h(y) \cdot  \partial^i_x F_x(y) \, d \mu(y) \,  d\mu(x)\, \\
  +\int_{M} u(x) \cdot {\rm div}(\partial^i)(x) \int_M 
   h(y) \cdot F_x(y) \, d \mu(y) \,  d\mu(x)\,.
\end{multline*}
Combining the last two displays, we obtain
\[
 \int_{M} u(x) \cdot \partial^i h(x) \,d\mu(x)~
=
\int_{M} u(x) \int_M 
   h(y) \cdot  \partial^i_x F_x(y) \, d \mu(y) \,  d\mu(x)\,.
\]
It follows that for each $i$, we have 
 $(\partial^i h)(x)= \langle h, \partial^i_x F_x \rangle$,
and hence $\nabla h(x)
=
\int h(y) \nabla_x {F_x(y)} \,d\mu(y)$. Now the argument may be completed as in the compact case.
\end{proof}

%%%%%%%%%%%%%%%%%%%%%%%%%%%%%%%%%%%%%%

\section{Examples} \label{sec:examples}

In this section, we compare our estimates to exact formulas in some cases where 
the spectral functions are known explicitly. 

\begin{eg}[Euclidean space]
Using the Fourier inversion formula, one finds that 
the spectral function for the Laplacian on $d$-dimensional Euclidean space 
is given by 
\[ e_{\lambda}(x,y)~ 
=~ (2 \pi)^{-d} \int_{0}^{\sqrt{\lambda}}  \left( \int_{\Sbb^{d-1}}
\cos\bigl((x-y) \cdot \rho \cdot \omega\bigr) \cdot  d\sigma(\omega) \right) \,  
 \rho^{d-1} \cdot d\rho \,,
\]
where $\Sbb^{d-1} \subset \Rbb^d$ is the sphere of radius $1$ 
and $d \sigma$ is the measure on $\Sbb^{d-1}$ 
induced from Lebesgue measure on $\Rbb^d$.
In particular, by setting $x=y$, we find that
\[ \ct_x(\lambda)~
=~ \frac{\omega_d}{(2\pi)^d} \cdot \lambda^{\frac{d}{2}}. 
\]
In comparison, Theorem \ref{thm:homogeneous-noncompact} gives 
\[  \ct_{x}(\lambda)~ 
\leq~  \frac{1}{\omega_d \int_{0}^{\pi/2} \theta^d \cdot \sin(2 \theta)\, d \theta} 
\cdot \lambda^{\frac{d}{2}}. 
\]
(Another representation of the full spectral function is 
\[
e_\lambda(x, y) ~=~ (2\pi r/\sqrt\lambda)^{-d/2}J_{d/2}(r\sqrt\lambda)
\,,
\]
where $J_k$ is the Bessel function of the first kind of order $k$ and $r := \|x - y\|$.)
\end{eg}

\begin{eg}[Spheres] \label{eg:spheres}
The eigenvalues for the Laplacian associated to the sphere $\Sbb^d$ 
of constant curvature $+1$ can be explicitly computed.  In particular, the $k^{{\rm th}}$ distinct nonzero eigenvalue is 
$\phi(k) := k \cdot (k +d-1)$ and 
\[ N\bigl(\phi(k)\bigr)~ =~ \binom{k+d-1}{d}~ +~  \binom{k+d}{d}\,.  \]
In the special case when $d=2$, we find that $N\bigl(\phi(k)\bigr)= (k+1)^2$, and so
\[  \ct(\lambda)~ 
=~  \lambda + \frac{1}{2} \cdot \left(1 + \sqrt{1+4 \lambda} \right)
 \]
when $\lambda$ is an eigenvalue of the Laplacian for $\Sbb^2$.

For $r\leq \pi$, the injectivity radius of $\Sbb^d$, we have 
\begin{equation} \label{eq:sphere-volume}
  V(r)~ =~ \vol(\Sbb^{d-1}) \cdot \int_0^r (\sin x)^{d-1}\, d  x\,.
\end{equation}
Thus, after a straightforward calculation we find that when $d=2$,
Theorem \ref{thm:transitiveeig} gives the estimate
\[  \ct(\lambda)~ 
\leq~ \frac{8 \lambda -2}{4 \lambda \cdot \sin^2\bigl(\pi/(4\sqrt{\lambda})\bigr) -1}  \]
for $\lambda >1/4$. See Figure \ref{fig:sphere-comparison}.

In comparison, in the case of $\Sbb^2$, 
Theorem 18 of \cite{Li-Yau} gives that
if $\lambda\geq 6$, then $\ct(\lambda) \leq 3 \cdot 
2^9 \cdot e^2 \cdot \lambda$, and 
Theorem 25 of \cite{Li-Yau} gives that
$\ct(\lambda) \leq 3 \cdot 2^{18}\cdot \pi^{-3} 
\cdot \exp(\frac{1}{3} \pi^2) \cdot \lambda$ for 
$\lambda\geq 2$. Note that Theorem 12 of \cite{Li-Yau} gives
that $\lambda_1 > 1/4$, Corollary 8 of \cite{Li}
and Remark \ref{remk:antipode} above each give $\lambda_1 \ge 1$,
whereas, in fact, $\lambda_1=2$.
\begin{figure} 
\centering
\includegraphics[width=4truein]{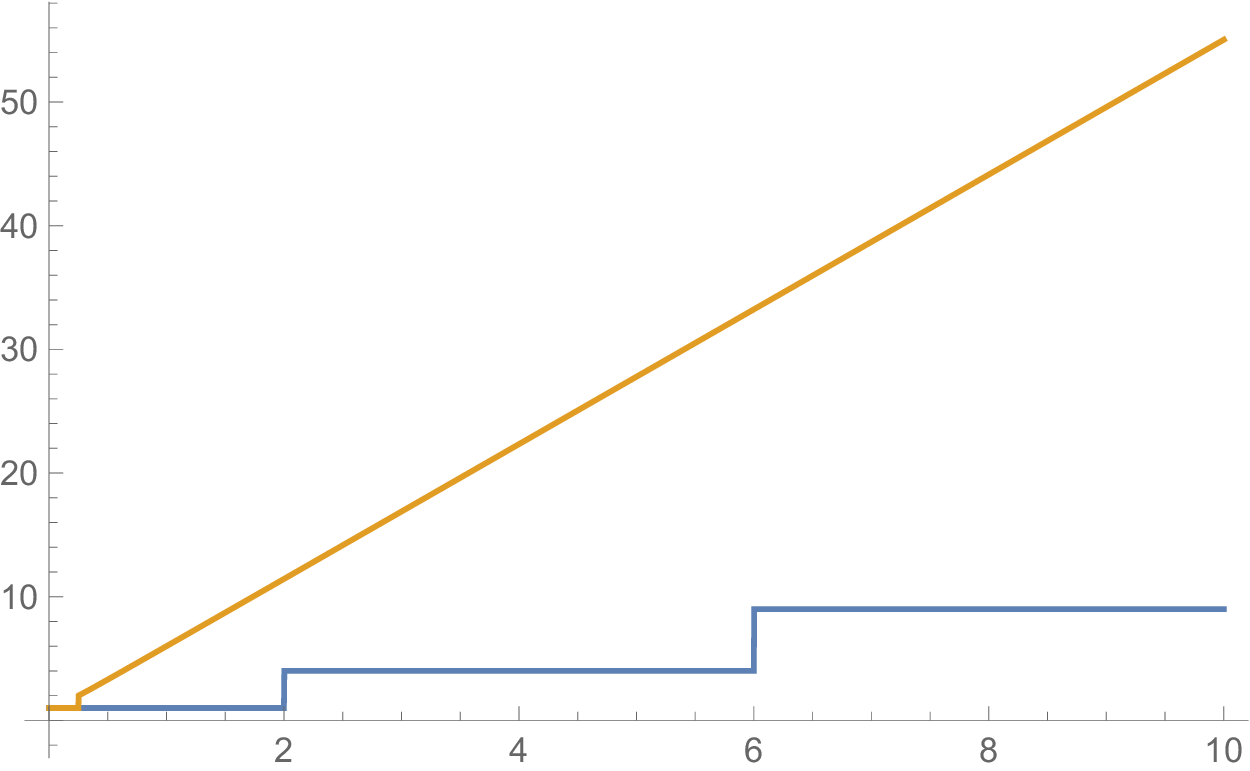}
\caption{The estimate of the counting function for $\Sbb^2$.
The piecewise constant function is $\ct(x)$ and the other function
is given by Theorem  \ref{thm:transitiveeig} and Remark \ref{remk:antipode}.
  \label{fig:sphere-comparison}}
\end{figure}
\end{eg}

\begin{eg}[Hyperbolic spaces]
Because the heat kernel is the Laplace transform of the 
spectral measure---see \eqref{eq:laplace}---the spectral function 
for a noncompact 
symmetric space $X$ can sometimes be determined if one has 
an exact expression for its heat kernel.\footnote{One can also
apply the inverse Fourier transform to an exact expression for the wave kernel. See \cite{LaxPhillips} 
for an explicit expression for the wave kernel of real 
hyperbolic space.}
Note that if  $X$ is a rank-one symmetric space, then 
$e_{\lambda}(x,y)$ depends only on the distance $r:=r(x,y)$
between $x$ and $y$.\footnote{Indeed, rank-one symmetric spaces
are ``two-point homogeneous": If $\dist(x,y)=\dist(x',y')$, 
then there exists an isometry
$\iota\colon M \to M$ with $\iota(x)=x'$ and $\iota(y)=y'$ \cite{Helgason}.
Thus, Lemma \ref{lem:e-invariance} implies that $e_{\lambda}$ 
depends only on $r$.}

For example, if $d\geq 3$ is odd, then, according to \cite{Anker}, 
the heat kernel for real hyperbolic space $\Hbb^d(\Rbb)$ is
\begin{equation} \label{eqn:hyperbolic-heat}
p_t(x,y)~
=~
\frac{\sqrt{\pi}}{(2 \pi)^{\frac{d+1}{2} }}
\cdot
\frac{ \exp\left(-b_d \cdot t\right) }{ \sqrt{t} }
\cdot 
\left( \frac{-1}{\sinh(r)} \partial_r \right)^{\frac{d-1}{2}}
\exp\left( \frac{-r^2}{4t} \right)\,,
\end{equation}
where $b_d := (d-1)^2/4$ is the bottom of the spectrum. 
We claim that if $r=\dist(x,y)>0$, then 
\begin{equation} \label{eqn:hyperbolic-spec}
e_\lambda(x, y)~ =~  \begin{cases}
\displaystyle
\frac{2}{(2 \pi)^{\frac{d+1}{2}} }
     \left( \frac{-1}{\sinh(r)} \partial_r \right)^{\frac{d-1}{2}} 
   \frac{\sin(r \cdot \sqrt{\lambda-b_d})}{r} &\text{if } \lambda \ge b_d,\\[3ex]
   0 &\text{otherwise}. 
   \end{cases}
\end{equation}
Indeed, by taking the Laplace transform of $de_{\lambda}$, making the change 
of variable $\lambda= \nu+b_d$,  and applying the identity
\[
\int_0^\infty e^{-\nu \cdot t} 
\,
d_{\nu}\left(\frac{\sin(r \sqrt{\nu})}{r} \right)
~=~
\frac{\sqrt\pi}{2\sqrt t} \cdot \exp\left( \frac{-r^2}{4t}\right)\,,
\]
we obtain (\ref{eqn:hyperbolic-heat}) from \eqref{eqn:hyperbolic-spec}. 
Thus, \eqref{eqn:hyperbolic-spec} follows from the fact that 
the Laplace transform is injective on bounded functions.

One obtains an  explicit expression for $N_{x}(\lambda)$
from  \eqref{eqn:hyperbolic-spec} by differentiating and 
taking the limit as $r$ tends to zero. 
For example, the values of $\ct_x(\lambda)$ for $d = 3, 5, 7$ 
and $\lambda \ge b_d$
are
%
%\begin{align*}
\[
\frac{(\lambda -1)^{3/2}}{6 \pi ^2}\,,\ \frac{(3 (\lambda -4)+5) (\lambda -4)^{3/2}}{180 \pi
   ^3}\,,\ \frac{\left(3 (\lambda -9)^2+21 (\lambda -9)+28\right) (\lambda -9)^{3/2}}{2520 \pi
   ^4}\,.
\]
   %\\[1.5ex]
   %&\frac{\left(5 (\lambda -16)^3+90 (\lambda -16)^2+441 (\lambda -16)+540\right) (\lambda
   %-16)^{3/2}}{75600 \pi ^5}\,,\\[1.5ex]
   %&\frac{\left(3 (\lambda -25)^4+110 (\lambda -25)^3+1287 (\lambda -25)^2+5412
   %(\lambda -25)+6336\right) (\lambda -25)^{3/2}}{997920 \pi ^6}\,.
%\end{align*}
%

For even-dimensional real hyperbolic spaces, 
a similar analysis using the explicit formula 
for the heat kernel \cite{Anker} gives
\[
e_\lambda(x, y)
~=~
\frac{2\sqrt 2}{(2\pi)^{(d+2)/2}} \int_r^\infty 
     \frac{\sinh s}{\sqrt{\cosh s -\cosh r}} \cdot 
     \left( \frac{-1}{\sinh s} \partial_s \right)^{\frac{d}{2}} 
     \frac{\sin (s \cdot \sqrt{\lambda - b_d})}{s}\,ds\,.
\]
It is convenient to write 
\[
\frac{\sin(s\cdot \sqrt{\lambda - b_d})}{s}
=
\int_0^{\sqrt{\lambda - b_d}} \cos(s \cdot a) \,da
\]
in the preceding integral. Using Fubini's theorem, we then obtain
for $d = 2$ and $\lambda \ge 1/4$ that 
\[
\ct_x(\lambda)
~=~
\int_0^{\sqrt{\lambda - 1/4}} \frac{a \tanh (\pi  a)}{2 \pi } \,da\,;
\]
when $d = 4$, we obtain for $\lambda \ge 9/4$ that
\[
\ct_x(\lambda)
~=~
\int_0^{\sqrt{\lambda - 9/4}} \frac{ \left(4 a^3+a\right) \tanh (\pi  a) }{32 \pi ^2 } \,da\,;
\]
and when $d = 6$, we obtain for $\lambda \ge 25/4$ that
\[
\ct_x(\lambda)
~=~
\int_0^{\sqrt{\lambda - 25/4}} \frac{\left(16 a^5+40 a^3+9a\right) \tanh (\pi  a)}{1024 \pi ^3} \,da\,.
\]

The volume of the ball of radius $r$
in $\Hbb^d(\R)$ is given by 
\[ V(r)~ =~  \vol(\Sbb^{d-1}) \int_{0}^r (\sinh x)^{d-1}\, dx\,. \]
In particular, if $d=2$, 
then Theorem \ref{thm:homogeneous-noncompact}
yields the estimate
\[ \ct_x(\lambda)~ 
\leq~  
\frac{4 \lambda +1}{2 \pi  \bigl(4 \lambda  \sinh^2 \bigl(\frac{\pi }{4 \sqrt{\lambda
   }}\bigr)-1\bigr)}\,.
\]
See Figure \ref{fig:hyperbolic-comparison}.

\begin{figure}

\centering
\includegraphics[width=4truein]{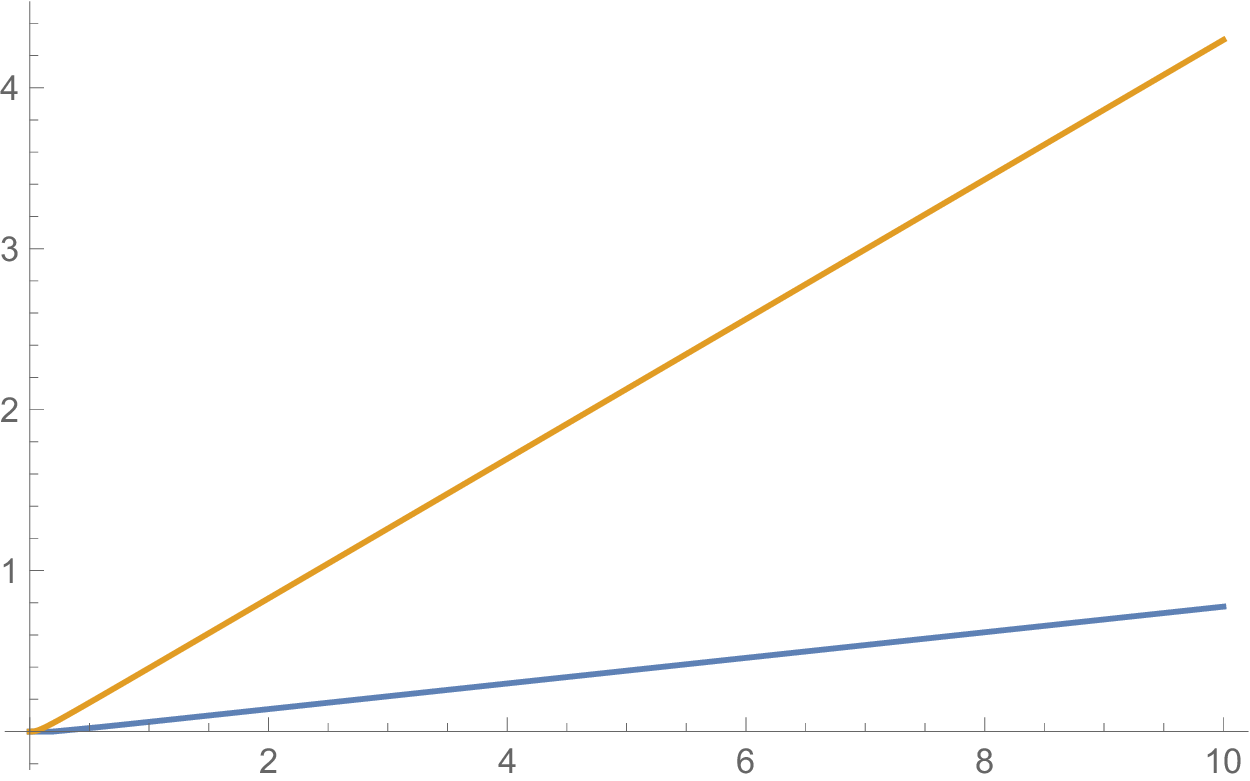}

\caption{The lower curve is $\ct_{x}$ for $\Hbb^2(\R)$ and 
the upper curve is the estimate given by Theorem
 \ref{thm:homogeneous-noncompact}.  
\label{fig:hyperbolic-comparison}}

\end{figure}

\end{eg}

Because we could not find the spectral functions for complex hyperbolic spaces in the literature, we give them here as well. 
For $M = \Hbb^d(\C)$ with $d \ge 2$, using the Laplace transform 
and the formula for the heat kernel \cite{Anker}, we obtain that
for $\lambda \ge d^2$ and $\dist(x, y) = r$, 
\[
e_\lambda(x, y) 
~=~
\frac{4\sqrt 2}{(2\pi)^{d+1}} \int_r^\infty 
     \frac{\sinh s}{\sqrt{\cosh 2s -\cosh 2r}} \cdot 
     \left( \frac{-1}{\sinh s} \partial_s \right)^{d} 
     \frac{\sin (s \cdot \sqrt{\lambda - d^2})}{s}\,ds\,.
\]
For $d = 2$, this yields
for $\lambda \ge 4$ that
\[
\ct_x(\lambda)
~=~
\int_0^{\sqrt{\lambda - 4}} \frac{a^3 \coth (\pi  a/2)}{8 \pi^2 } \,da\,;
\]
for $d = 3$, this yields
for $\lambda \ge 9$ that
\[
\ct_x(\lambda)
~=~
\int_0^{\sqrt{\lambda - 9}} \frac{a(a^2+1)^2 \tanh (\pi  a/2)}{64 \pi^3 } \,da\,;
\]
and for $d = 4$, this yields
for $\lambda \ge 16$ that
\[
\ct_x(\lambda)
~=~
\int_0^{\sqrt{\lambda - 16}} \frac{a^3(a^2+4)^2 \coth (\pi  a/2)}{768 \pi^4 } \,da\,.
\]

\iffalse
For $M = H^d(\Hbb)$ with $d \ge 2$, it follows that
for $\lambda \ge (2d+1)^2$ and $\dist(x, y) = r$,
\begin{align*}
e_\lambda(x, y) 
~=~
4 \int_r^\infty 
%
     \frac{\sinh s}{\sqrt{\cosh 2s -\cosh 2r}} &\cdot 
     \frac{-1}{2 \pi \sinh s} \partial_s 
     \frac{-1}{2 \pi \sinh 2s} \partial_s \cdot \\
     &\cdot \left( \frac{-1}{2 \pi \sinh s} \partial_s \right)^{2d-2} 
%
     \frac{\sin (s \cdot \sqrt{\lambda - (2d+1)^2})}{\pi s}\,ds\,.
\end{align*}
%For $d = 2$, this yields
%for $\lambda \ge 25$ that
%\[
%\ct_x(\lambda)
%=
%\int_0^{\sqrt{\lambda - 25}} \frac{a^3 \coth (\pi  a/2)}{8 \pi^2 } \,da\,;
%\]
%for $d = 3$, this yields
%for $\lambda \ge 9$ that
%\[
%\ct_x(\lambda)
%=
%\int_0^{\sqrt{\lambda - 9}} \frac{a(a^2+1)^2 \tanh (\pi  a/2)}{64 \pi^3 } \,da\,;
%\]
%and for $d = 4$, this yields
%for $\lambda \ge 16$ that
%\[
%\ct_x(\lambda)
%=
%\int_0^{\sqrt{\lambda - 16}} \frac{a^3(a^2+4)^2 \coth (\pi  a/2)}{768 \pi^4 } \,da\,.
%\]
For $M = H^2(\Obb)$, it follows that
for $\lambda \ge 11^2$ and $\dist(x, y) = r$,
\begin{align*}
e_\lambda(x, y) 
~=~
2^{7/2} \int_r^\infty 
%
     \frac{\sinh s}{\sqrt{\cosh 2s -\cosh 2r}} &\cdot 
     \frac{-1}{2 \pi \sinh s} \partial_s 
     \left(\frac{-1}{2 \pi \sinh 2s} \partial_s\right)^3 \cdot \\
     &\cdot \left( \frac{-1}{2 \pi \sinh s} \partial_s \right)^{4} 
%
     \frac{\sin (s \cdot \sqrt{\lambda - 11^2})}{\pi s}\,ds\,.
\end{align*}
\fi

%%%%%%%%%%%%%%%%%%%%%%%%%%%%%%%%%%%%%%%%

\appendix
\section{Integral kernels and the functional calculus}
\label{sec:appendix}

It is known that the spectral projection 
$E_\lambda$ has a smooth kernel in the case of noncompact 
manifolds; see, for example, \cite{Hormander66}. However, 
we found it difficult to find 
a complete proof in the literature, 
and so we provide a proof here as a courtesy to the reader. 
We also prove some other facts about the kernel that we use.

The proof of smoothness relies on the standard theory of elliptic partial differential equations. If $\Delta^k u =v$, then since $\Delta$ is self-adjoint, 
we have, for each compactly supported smooth function $\psi$,
\begin{equation} \label{eq:distribution}
 \langle \Delta^k \psi, u\rangle~ =~  \langle  \psi, v \rangle
 \,.
\end{equation}
That is, if we regard $u$ and $v$ as distributions,  
then $L^k u = v$, where $L$ is the Laplacian acting on distributions.
Here, $L^k$ is an elliptic differential operator of order $2k$.
Our assumption $\Delta^k u =v$ involves a function class $u \in \Hcal = L^2(M, \mu)$,
but we will need to show that there is a function in the class of $u$ that
has continuous derivatives.
Thus, we explain next a few facts about Sobolev spaces that we need, which is a
theory of distributions.

Let $U \subset M$ be a precompact open set whose closure 
lies in an open set diffeomorphic to $\Rbb^d$. A function (class) $v \in
\Hcal$, regarded as a distribution restricted to $U$,
lies in the Sobolev space $H^0(U)$. If $u$ is a distribution with $L^k u =v$, then 
the standard theory of elliptic equations implies that 
$u$ lies in the Sobolev space $H^{2k}(U)$.\footnote{See, 
for example, the proof of Theorem 1 in Chapter 4 of \cite{BJS}.}
Moreover, there exists a constant $C=C(L,U,k)$ 
so that 
\begin{equation} \label{eq:elliptic-estimate}
    \|u \|_{H^{2k}(U)}~ \leq~ 
    C \cdot \bigl( \|v\| + \|u\| \bigr)\,.  
\end{equation}

Choose coordinates $(x_{1}, x_2, \ldots, x_d)$ on $U$,
and for each multi-index $\alpha \in \Nbb^d$, let 
$\partial^{\alpha}$ denote the mixed partial derivative
$\partial^{\alpha_1}_{x_1} \cdots \partial^{\alpha_d}_{x_d}$.
For each $u\colon U \rightarrow \Rbb$ such that 
$\partial^{\alpha} u$ is continuous and bounded for each 
$\alpha$ with $|\alpha|:=\sum_{} \alpha_i \leq j$, define  
\begin{equation} \label{eq:Sobolev}
 \|u\|_{C^j(U)}~ =~
    \sup_{U}\, |u|_{C^j}\,,
\end{equation}
where
\begin{equation*} 
|u|_{C^j}(x)~ :=~
   \sup_{|\alpha|\leq j}\,   \left|\partial^{\alpha} u(x) \right|\,. 
\end{equation*}
The Sobolev embedding theorem provides
constants $C'=C'(U,k,j)$ so that if $u \in H^{2k}(U)$ 
and $2k > d/2 +j+1$, then the distribution $u $ is represented
by a function with a continuous $j^{{\rm th}}$ derivative, which we also denote by $u$, and 
\[    \|u \|_{C^j(U)}~ \le~ C' \cdot \| u \|_{H^{2k}(U)}\,.  \]

Recall that $E$ denotes the spectral resolution of $\Delta$.

\begin{thm} \label{lem:kernel}
Let $a \colon [0, \infty) \to \R$ be a nonnegative Borel function
such that for each nonnegative integer $k$, the map 
$\nu \mapsto a(\nu) \cdot \nu^k$ is bounded
on $[0, \infty)$.  Then the operator $a(\Delta)$ defined by 
\[   a(\Delta)~ =~ \int_0^{\infty} a(\nu)~ dE_{\nu} 
\] 
maps each element of $\Hcal$ to a smooth function, and
$a(\Delta)$ has a smooth, symmetric, real-valued integral kernel. 
\end{thm}

\begin{proof}
The hypothesis implies that, for each $k$, the operator 
$\Delta^k \circ \sqrt{a}(\Delta)$ maps $\Hcal$ to $\Hcal$. 
Given  $w \in \Hcal$, let $u:=\sqrt{a}(\Delta)(w)$ and $v:= \Delta^k u$.  
Regarding $u$ and $v$ as distributions,
we have $L^k u = v$. Given $x \in M$, let $U \ni x$ be a precompact open set
whose closure lies in an open set diffeomorphic to $\Rbb^d$. 
By the discussion above,
the distribution $u$ lies in $H^{2k}(U)$ and (\ref{eq:elliptic-estimate})
holds, and if we choose $2k> d/2 + j+ 1$, then 
\[ |u|_{C^{j}}(x)~ \leq~ C' \cdot \|u\|_{H^{2k}(U)}~ \le~
   C' \cdot C \cdot \bigl( \|v\| + \|u\| \bigr)\,,
\]
and hence, since $\Delta^k \circ \sqrt{a}(\Delta)$ is a bounded 
linear operator, we have 
\begin{equation}   \label{eq:bounded-linear-functional}
  \left|(\sqrt{a}(\Delta))w \right|_{C^{j}}(x)~ \leq~ 
   C' \cdot C \cdot \left( \|\Delta^k \circ \sqrt{a}(\Delta) w\|~
 + \|w\| \right)~
  \leq~ C^* \cdot \|w\| 
\end{equation}
for some constant $C^*=C^*(U,j,k)$. In particular, $\sqrt a(\Delta) w$ is smooth.

For each $x$, define the linear functional $f_x\colon \Hcal \to \Rbb$ 
by $f_x(w) := (\sqrt{a}(\Delta) w)(x)$. 
From (\ref{eq:bounded-linear-functional}) with $j=0$, 
we see that the functional $f_x$ 
is bounded, and therefore there exists $k_x \in \Hcal$ so that 
for each $w \in \Hcal$ we have 
$\langle w, k_x \rangle = f_x(w)=(\sqrt{a}(\Delta) w)(x)$. 

We claim that $k_x(y)=k_y(x)$ for almost every $(x,y) \in M \times M$. 
Indeed, since $\sqrt{a}(\Delta)$ is self-adjoint, $k_x$ is real valued and $\langle \sqrt{a}(\Delta) w_-, \overline{w_+} \rangle = 
\langle  w_-, \sqrt{a}(\Delta)\overline{w_+} \rangle$ for each $w_-, w_+ \in \Hcal$. Hence 
\begin{multline*}
\int_{M \times M} w_-(x) \cdot w_+(y) \cdot k_x(y)\, (d\mu \times d\mu) (x,y)\\
=    \int_{M \times M} w_+(y) \cdot w_-(x) \cdot k_x(y)\, (d\mu \times d\mu) (x,y)\,.
\end{multline*}
By switching the roles of $x$ and $y$ in the latter integral and subtracting, 
we find that
\[  \int_{M \times M} \psi(x, y) \cdot 
    \bigl(k_x(y) -k_y(x) \bigr)\, (d\mu \times d\mu) (x,y)~ =~ 0\,,
\]
where $\psi(x,y) := w_-(x) \cdot w_+(y)$. Since such $\psi$ span a dense subspace of $L^2(M \times M, \mu \times \mu)$, the claim follows.

Define $K\colon M \times M \to \Rbb$ by
\[ K(x,y)~ :=~  \int_{M} k_x(z) \cdot k_y(z)\, \dmu(z)\,.  \]
Using the symmetry $k_y(z)=k_z(y)$ $\>\mu \times \mu$-a.e.\ and the fact that 
$\sqrt{a}\cdot \sqrt{a}=a$, one finds
that $K$ is an integral kernel for $a(\Delta)$.  
To finish the proof of the theorem, it suffices to show 
that $K$ is a smooth function. 

Let $\alpha \in \Nbb^d$ be a multi-index. Let $U_x$ and $U_y$ be neighborhoods of $x, y \in M$.
Using (\ref{eq:bounded-linear-functional}), we find that the map 
$ w \mapsto \partial^{\alpha}_x (\sqrt{a}(\Delta)w)(x)$ is 
a bounded linear functional on $\Hcal$. Hence there
exists $k_x^{\alpha} \in \Hcal$ so that 
\[ \langle k_x^{\alpha}, w \rangle~ =~ 
\partial^{\alpha}_x (\sqrt{a}(\Delta) w)(x)~
=~  \partial_{x}^{\alpha} \langle k_x, w \rangle\,.
\]
Thus, we have 
\[ \partial^{\beta}_y \partial^{\alpha}_x K(x,y)~ 
=~  \partial^{\beta}_y \partial_{x}^{\alpha} \langle k_x, k_y \rangle~
=~ \partial^{\beta}_y  \langle k_x^{\alpha}, k_y \rangle~
=~ \langle k_x^{\alpha}, k_y^{\beta} \rangle\,. \]
In particular, $K$ is smooth. 
\end{proof}

Let $e_\lambda(x, y)$ denote the kernel corresponding to the spectral projection $E_\lambda = \b1_{[0, \lambda]}(\Delta)$. 
Write $F_x^\lambda(y) := e_\lambda(x, y)$, so that $F_x^\lambda$ is an element of $\Hcal$ with norm 1.
For continuous $a \colon [0, \infty) \to \R$ with support in $[0, \lambda]$
and all $\nu \ge \lambda$, we have $\iprod{a(\Delta) F_x^\nu, F_y^\nu} = \iprod{a(\Delta) F_x^\lambda, F_y^\lambda}$, with
$|\iprod{a(\Delta) F_x^\lambda, F_y^\lambda}| \le \|a\|_\infty$ by the Cauchy--Schwarz inequality. Therefore,
$\lambda \mapsto e_\lambda(x, y)$ is a function of bounded variation and defines a bounded measure on the real line that
satisfies
\[
\int_0^\infty a(\nu) \,de_\nu(x, y)
~=~
\iprod{a(\Delta) F_x^\lambda, F_y^\lambda}
\]
whenever $a$ is a bounded Borel function with support in $[0, \lambda]$.

\begin{prop} \label{lem:kernel2}
Let $a \colon [0, \infty) \to \R$ be a nonnegative Borel function
such that for each nonnegative integer $k$, the map 
$\nu \mapsto a(\nu) \cdot \nu^k$ is bounded
on $[0, \infty)$.  
Then the kernel of $a(\Delta)$ is
\[
(x, y) \mapsto A_x(y) ~:=~ \int_0^\infty a(\lambda) \, de_\lambda(x, y)
\,.
\]
\end{prop}

\begin{proof}
Let $K$ be the kernel given by Theorem \ref{lem:kernel} and $K_x(y) := K(x, y)$.
Then for $x \in M$ and $u \in \Hcal_\nu$, we have $a(\Delta)u \in \Hcal_\nu$ and
\begin{align*}
\iprod{u, K_x}
&~=~
a(\Delta)u(x)
~=~
\Bigiprod{\int_0^\nu a(\lambda) \,dE_\lambda u, F_x^\nu}
~=~
\int_0^\nu a(\lambda) \,d\iprod{E_\lambda u, F_x^\nu}
\\ &~=~
\int_0^\nu a(\lambda) \,d\iprod{u, F_x^\lambda}
~=~
\Bigiprod{u, \int_0^\nu a(\lambda) \,dF_x^\lambda}
~=~
\iprod{u, A_x}
\,.
\end{align*}
Since $\iprod{u, K_x} = \iprod{u, A_x}$ for such $u$ and $\bigcup_\nu \Hcal_\nu$ is
dense in $\Hcal$, we obtain that $K_x = A_x$, as desired.
\end{proof}

\begin{prop} \label{prop:gradient-estimate}
Let $K$ be the kernel of $a(\Delta)$ constructed in Theorem \ref{lem:kernel}.
For each $x$, the function $y \mapsto |\nabla_x K(x,y)|^2$ is integrable.  
\end{prop}

\begin{proof}
Choose normal coordinates $(x_1, \ldots, x_d)$ about the point $x$, 
and let $\partial_i$ be the associated coordinate vector fields. 
By (\ref{eq:bounded-linear-functional}), for each $i$, there exists $k_x^i \in \Hcal$ so that
$\langle k_x^i, u\rangle = \partial_i \langle k_x,u \rangle$
for each $u \in \Hcal$. 
Thus, 
\begin{equation*}
|\nabla_x K(x,y)|^2~ 
=~ \sum_i |\partial_i \langle k_x, k_y \rangle |^2~
  =~ \sum_i \langle k_x^i, k_y \rangle^2.
\end{equation*}
On the other hand, the defining property of $k_y$ yields
$\langle k_x^i, k_y \rangle = (\sqrt{a}(\Delta)k_x^i)(y)$. 
Since $\sqrt{a}(\Delta)$ maps $\Hcal$ to $\Hcal$, the 
function $y \mapsto (\sqrt{a}(\Delta)k_x^i)(y)$ is square-integrable. The claim 
follows from summing over $i$.
\end{proof}

%%%%%%%%%%%%%%%%%%%%%%%%%%%%%%%%%
% Bibliography
%%%%%%%%%%%%%%%%%%%%%%%%%%%%%%%%%

%%%%%%%%%%%%%%%%%%%%%%%%%%%%%%%%%
% Projects and funding
% Uncomment and fill in with the information if needed
%%%%%%%%%%%%%%%%%%%%%%%%%%%%%%%%%

\projects{\noindent 
The work of C.J.\ is partially supported by a Simons collaboration grant.
The work of R.L.\ is partially supported by the National
Science Foundation under grants DMS-1007244 and DMS-1612363.}

\end{document}